\documentclass{article}
\usepackage[english]{babel}
\usepackage{amsmath}
\usepackage{amsthm}
\usepackage{amssymb}
\usepackage{amscd}
\usepackage[latin1]{inputenc}
\usepackage[all]{xy}

\newtheorem{teor}{Theorem}[section]
\newtheorem{defi}{Definition}
\newtheorem{lemma}[teor]{Lemma}
\newtheorem{prop}[teor]{Proposition}
\newtheorem{cor}[teor]{Corollary}
\newtheorem{rem}[teor]{Remark}

\newtheorem{exem}[teor]{Example}

\newtheorem{ques}[teor]{Question}

\topskip=\baselineskip  \headsep=0.3in \headheight=0in
\title{On hearts which are module categories}
\author{Carlos E. Parra \thanks{The authors thank Pere Ara for his comments
on the surjectivity of the monoid map $V(R)\longrightarrow
V(R/\mathbf{a})$ and for telling us about the reference \cite{A}. Parra is supported by a grant from the Universidad de los Andes (Venezuela) and Saor\'in is
supported by research projects from the Spanish Ministry of
Education  and from the Fundaci\'on 'S\'eneca'
of Murcia, with a part of FEDER funds. The authors
thank these institutions for their help.} \\
Departamento de Matem\'aticas\\ Universidad de los Andes \\ ({\bf 5101}) M\'erida\\ VENEZUELA\\
{\it carlosparra@ula.ve} \\  \\ Manuel Saor\'in \footnotemark[1]  \\ Departamento de Matem\'aticas\\
Universidad de Murcia, Aptdo. 4021\\
30100 Espinardo, Murcia\\
SPAIN\\ {\it msaorinc@um.es}  }
\begin{document}
\date{}
\maketitle

\begin{abstract}
{\bf Given a torsion pair $\mathbf{t}=(\mathcal{T},\mathcal{F})$ in
a module category $R-\text{Mod}$ we give necessary and sufficient
conditions for the associated Happel-Reiten-Smal\o \  t-structure in
$\mathcal{D}(R)$ to have a heart $\mathcal{H}_\mathbf{t}$ which is a
module category. We also study when such a pair is given by a 2-term
complex of projective modules in the way described by
Hoshino-Kato-Miyachi (\cite{HKM}). Among other consequences, we
completely identify  the hereditary torsion pairs $\mathbf{t}$ for
which $\mathcal{H}_\mathbf{t}$ is a module category in the following
cases: i) when $\mathbf{t}$ is the left constituent of a TTF triple,
showing that $\mathbf{t}$ need not be HKM; ii) when $\mathbf{t}$ is
faithful; iii) when $\mathbf{t}$ is arbitrary and the ring $R$ is
either commutative, semi-hereditary, local, perfect or Artinian. We
also give a systematic way of constructing non-tilting torsion pairs
 for which the heart is a module category
generated by  a stalk complex at zero}

\end{abstract}

\hspace*{0.5cm}

{\bf Mathematics Subjects Classification: 16Exx, 18Gxx, 16B50}

\section{Introduction}
Beilinson, Bernstein and Deligne \cite{BBD} introduced the notion of
 t-structure in a triangulated category in their study of perverse
sheaves on an algebraic or analytic variety. If $\mathcal{D}$ is
such a triangulated category, a t-structure in $\mathcal{D}$ is a
pair of full subcategories satisfying suitable axioms (see the
precise definition in next section) which guarantee that their
intersection is an abelian category $\mathcal{H}$, called the heart
of the t-structure. This category  comes with a cohomological
functor $\mathcal{D}\longrightarrow\mathcal{H}$. Roughly speaking, a
t-structure allows to develop an intrinsic (co)homology theory,
where the homology 'spaces' are again objects of $\mathcal{D}$
itself.

In the context of bounded derived categories, Happel, Reiten and
Smal\o \ \cite{HRS} associated to each torsion pair
$\mathbf{t}$ in an abelian category $\mathcal{A}$, a t-structure in
the bounded derived category $\mathcal{D}^b(\mathcal{A})$. This
t-structure is actually the restriction of a t-structure
 in the unbounded derived category $\mathcal{D}(\mathcal{A})$, when this later category is defined.
 Several authors (see \cite{CGM}, \cite{CMT}, \cite{MT}, \cite{CG})
have dealt with the problem of deciding when its heart
$\mathcal{H}_\mathbf{t}$  is a Grothendieck or module category. When
$\mathcal{A}=\mathcal{G}$ is a Grothendieck category, after the
recent work by the authors (see \cite{PS}), it seems that the
condition that $\mathcal{H}_\mathbf{t}$ be a Grothendieck category
is well understood. Indeed, under fairly general hypotheses,
$\mathcal{H}_\mathbf{t}$ is a Grothendieck category if, and only if,
the torsionfree class of the pair is closed under taking direct
limits in $\mathcal{G}$ (see \cite[Theorem 4.9]{PS}).

The situation when $\mathcal{H}_\mathbf{t}$ is a module category is
far less understood, even in the case when
$\mathcal{A}=R-\text{Mod}$ is a module category. The problem has
been tackled, from different perspectives, in \cite{HKM},
\cite{CGM}, \cite{CMT} and \cite{MT}. In the second of these
references, the authors show that an abelian category with a
classical 1-tilting object is equivalent to
$\mathcal{H}_\mathbf{t}$, for some faithful torsion pair
$\mathbf{t}$ in a module category. Since a classical tilting object
defines an equivalence between the derived categories of the ambient
abelian category and  of the endomorphism ring of the object,
faithful torsion pairs in module categories became natural
candidates to study when the heart is a module category. In
\cite{CMT} the authors pursued this line and gave necessary and
sufficient conditions for a faithful torsion pair in a module
category to have a modular heart. In the earlier paper \cite{HKM},
the authors had associated a pair of subcategories
$(\mathcal{X}(P^\bullet ),\mathcal{Y}(P^\bullet ))$ of
$R-\text{Mod}$ to a $2$-term  complex $P^\bullet$ of finitely
generated projective modules. Then they gave necessary and
sufficient conditions for the pair to be a torsion pair, in which
case the corresponding heart was a module category. In \cite{MT},
for a given torsion pair $\mathbf{t}$ in $R-\text{Mod}$, the authors
compared the conditions that the heart be a module category with the
condition that $\mathbf{t}$ be a torsion pair as in \cite{HKM}. In
particular, they proved that if $\mathbf{t}$ is faithful then both
conditions were equivalent.

In the present paper, given any torsion pair $\mathbf{t}$ in a
module category $R-\text{Mod}$, we  give necessary and sufficient
conditions for the heart $\mathcal{H}_\mathbf{t}$ to be a module
category and, simultaneously,  compare this property with that of
$\mathbf{t}$ being an HKM torsion pair (see next section for all the
pertinent definitions of the terms that we use in this
introduction). When tackled in full generality, the conditions that
appear tend to be rather technical, but a deeper look in particular
cases gives more precise information on the torsion pair. Roughly
speaking, when one assumes that $\mathbf{t}$ is hereditary one falls
into the world of TTF triples, while if one assumes that the torsion
class is closed under taking products in $R-\text{Mod}$, then one
enters the world of classical tilting torsion pairs.

The following is a list of the main results, all of them given for a
torsion pair $\mathbf{t}=(\mathcal{T},\mathcal{F})$ in
$R-\text{Mod}$:

\begin{enumerate}
\item (Part of theorem \ref{teor.hereditary case}) If $\mathbf{t}$
is hereditary and $\mathcal{H}_\mathbf{t}$ is a module category,
then
$\mathbf{t}'=(\mathcal{T}\cap\frac{R}{t(R)}-\text{Mod},\mathcal{F})$
is the right constituent torsion pair of a TTF triple in
$\frac{R}{t(R)}-\text{Mod}$. When $\mathbf{t}$ is bounded, it is
itself the right constituent pair of a TTF triple in $R-\text{Mod}$.

\item (Corollary \ref{cor.progenerator stalk in 0}) $\mathcal{H}_\mathbf{t}$  has a progenerator which is a stalk
complex $V[0]$ if, and only if, $\mathbf{t}$ is the torsion pair
associated to a finitely presented quasi-tilted $R$-module $V$ such
that $\text{Ext}_R^2(V,?)_{| \mathcal{F}}=0$ and $\mathcal{T}$
cogenerates $\mathcal{F}$. There is a systematic way (see theorem
\ref{teor.nontilting pair with stalk progenerator}) of constructing
non-tilting modules $V$ satisfying this property.

\item (Part of proposition \ref{prop.non-HKM modular heart}) If
$\mathbf{t}$ is hereditary and the left constituent pair of a TTF
triple, then $\mathcal{H}_\mathbf{t}$ is a module category if, and
only if, there is a finitely generated projective module $P$ such
that $\mathcal{T}=\text{Gen}(P)$. In general, $\mathbf{t}$ need not
be HKM.

\item (Part of theorem \ref{teor.closed under products case}) If
$\mathcal{T}$ is closed under taking products in $R-\text{Mod}$ and
$\mathcal{H}_\mathbf{t}$ is a module category, then there is a
finitely presented module $V$ such that $\mathcal{T}=\text{Gen}(V)$
and $V$ is classical 1-tilting over $R/\mathbf{a}$, where
$\mathbf{a}=\text{ann}_R(V)$. Moreover, the torsion pair
$\mathbf{t}'=(\text{Gen}(V),\mathcal{F}\cap\frac{R}{\mathbf{a}}-\text{Mod})$
 in $R/\mathbf{a}-\text{Mod}$ has a  heart which is a module category and embeds faithfully in
$\mathcal{H}_\mathbf{t}$.

\item (Theorem \ref{teor.modular H with easy conditions}) Suppose
that $\mathbf{t}$ is the right constituent pair of the TTF triple
$(\mathcal{C},\mathcal{T},\mathcal{F})$ in $R-\text{Mod}$ defined
by the idempotent ideal $\mathbf{a}$. Under fairly general
hypotheses, the heart $\mathcal{H}_\mathbf{t}$ is a module category
if, and only if, $\mathbf{a}$ is finitely generated on the left and
there is a finitely generated projective $R$-module $P$ such that:

\begin{enumerate}
\item $P/\mathbf{a}P$ is a progenerator of
$R/\mathbf{a}-\text{Mod}$;
\item There is an exact sequence $0\rightarrow F\longrightarrow C\longrightarrow\mathbf{a}P\rightarrow
0$, with $C$ finitely generated module in $\mathcal{C}$,  such that
$\text{Ext}_R^1(C,?)_{| \mathcal{F}}=0$ and $C$ generates
$\mathcal{C}\cap\mathcal{F}$.
\end{enumerate}
\item If $\mathbf{t}$ is the right constituent of the TTF triple
defined by a finitely generated projective module whose trace in $R$
is finitely generated, then $\mathcal{H}_\mathbf{t}$ is a module
category (corollary \ref{cor.torsion theory given by projective}).
Under fairly general hypotheses,    the converse is also true for
arbitrary faithful hereditary torsion pairs (corollary
\ref{cor.faithful hereditary torsion pair}).
\item For the following classes of rings, all hereditary torsion
pairs whose heart is a module category are identified: commutative
(corollary \ref{cor.hereditary pairs in commutative}),
semihereditary (proposition \ref{prop.semihereditary}), local,
perfect and artinian (corollary \ref{cor.local and artinian}).
\end{enumerate}

The organization of the paper goes as follows. Section 2 gives the
preliminaries that are needed and the terminology which is used in
the paper. Section 3 is devoted to giving necessary and sufficient
conditions on an arbitrary torsion pair $\mathbf{t}$ in
$R-\text{Mod}$ for its heart to be a module category and also for it
to be an HKM pair. In section 4 we assume that $\mathbf{t}$ is
hereditary and show how TTF triples appear naturally. In section 5,
we give necessary and sufficient conditions for
$\mathcal{H}_\mathbf{t}$ to have a progenerator which is a sum of
stalk complexes. In section 6 we assume that the torsion class is
closed under taking products and show that the modular condition on
$\mathcal{H}_\mathbf{t}$ naturally leads to classical tilting
torsion pairs. In  section 7, we assume that $\mathbf{t}$ is the
right constituent torsion pair of a TTF triple, and give necessary
and sufficient conditions for $\mathcal{H}_\mathbf{t}$ to be a
module category and for  $\mathbf{t}$ to be an HKM pair. We end the
paper with a final section of illustrative examples.

\section{Terminology and preliminaries} \label{sec.Terminology}
In this paper all rings are supposed to be associative with unit and
their modules will be always unital modules. Unless otherwise
stated, 'module' will mean 'left module' and if $R$ is a ring, we
shall denote by $R-\text{Mod}$ and $\text{Mod}-R$
(=$R^{op}-\text{Mod}$) its categories of left and right modules,
respectively. A \emph{module category} is any one which is
equivalent to $R-\text{Mod}$, for some ring $R$.

 The concepts that we shall
introduce in this section are mainly applied to the case of  module
categories, but sometimes we will use them in the most general
context of Grothendieck categories and is in this context that we
introduce them. Let then $\mathcal{G}$ be a Grothendieck category
all throughout this section.

A \emph{torsion pair} in  $\mathcal{G}$ is a pair
$\mathbf{t}=(\mathcal{T},\mathcal{F})$ of full subcategories
satisfying the following two conditions:

\begin{enumerate}
\item[] - $\text{Hom}_\mathcal{G}(T,F)=0$, for all $T\in\mathcal{T}$
and $F\in\mathcal{F}$;
\item[] - For each object $X$ of $\mathcal{G}$ there is an exact
sequence $0\rightarrow T_X\longrightarrow X\longrightarrow
F_X\rightarrow 0$, where $T_X\in\mathcal{T}$ and
$F_X\in\mathcal{F}$.
\end{enumerate}
In such case the objects $T_X$ and $F_X$ are uniquely determined, up
to isomorphism, and the assignment $X\rightsquigarrow T_X$ (resp.
$X\rightsquigarrow F_X$) underlies a functor
$t:\mathcal{G}\longrightarrow\mathcal{T}$ (resp.
$(1:t):\mathcal{G}\longrightarrow\mathcal{F}$) which is right (resp.
left) adjoint to the inclusion functor
$\mathcal{T}\hookrightarrow\mathcal{G}$ (resp.
$\mathcal{F}\hookrightarrow\mathcal{G}$). We will frequently write
$X/t(X)$ to denote $(1:t)(X)$. The composition
$\mathcal{G}\stackrel{t}{\longrightarrow}\mathcal{T}\hookrightarrow\mathcal{G}$,
which we will still denote by $t$,  is  called the \emph{torsion
radical} associated to $\mathbf{t}$.
 We call $\mathcal{T}$ and
$\mathcal{F}$ the \emph{torsion class} and \emph{torsionfree class}
of the pair, respectively. For each class $\mathcal{X}$ of objects,
we will put $\mathcal{X}^\perp =\{M\in\mathcal{G}:$
$\text{Hom}_\mathcal{G}(X,M)=0\text{, for all }X\in\mathcal{X}\}$
and  ${}^\perp\mathcal{X} =\{M\in\mathcal{G}:$
$\text{Hom}_\mathcal{G}(M,X)=0\text{, for all }X\in\mathcal{X}\}$.
If $\mathbf{t}$ is a torsion pair as above, then
$\mathcal{T}={}^\perp\mathcal{F}$ and
$\mathcal{F}=\mathcal{T}^\perp$. The torsion pair  is called
\emph{hereditary} when $\mathcal{T}$ is closed under taking
subobjects in $\mathcal{G}$. It is called \emph{split} when $t(X)$
is a direct summand of $X$, for each object $X$ of $\mathcal{G}$. If
$R$ is a ring and $\mathcal{G}=R-\text{Mod}$,   we will say that
$\mathbf{t}$ is \emph{faithful} when $R\in\mathcal{F}$.

A class $\mathcal{T}\subseteq\mathcal{G}$ is a \emph{TTF
(=torsion-torsionfree) class} when it is both a torsion and a
torsionfree class in $\mathcal{G}$. Each triple of the form
$(\mathcal{C},\mathcal{T},\mathcal{F})=({}^\perp\mathcal{T},\mathcal{T},\mathcal{T}^\perp)$,
for some TTF class $\mathcal{T}$, will be called a \emph{TTF triple}
and the two torsion pairs $(\mathcal{C},\mathcal{T})$ and
$(\mathcal{T},\mathcal{F})$ will be called the \emph{left
constituent pair} and \emph{right constituent pair} of the TTF
triple. The TTF triple is called \emph{left (resp. right) split}
when its left (resp. right) constituent torsion pair is split. It is
called \emph{centrally split} when both constituent torsion pairs
are split. When $\mathcal{G}=R-\text{Mod}$, it is well-known (see
\cite[Chapter VI]{S}) that  $\mathcal{T}$ is a TTF class if, and
only if,  there is a (unique) idempotent two-sided ideal $\mathbf{a}$ of $R$
such that $\mathcal{T}$ consists of the $R$-modules $T$ such that
$\mathbf{a}T=0$. Moreover, the torsion radical $c$ with respect to
$(\mathcal{C},\mathcal{T})$  assigns to each module $M$ the
submodule $c(M)=\mathbf{a}M$. In particular, we have
$\mathcal{C}=\text{Gen}(\mathbf{a})=\{C\in R-\text{Mod}:$
$\mathbf{a}C=C\}$. When $P$ is projective $R$-module,
$\mathcal{T}=\text{Ker}(\text{Hom}_R(P,?))$ is a TTF class and
$\text{Gen}(P)={}^\perp\mathcal{T}$. The corresponding idempotent
ideal is the trace of $P$ in $R$.

Given any additive category $\mathcal{A}$ with coproducts, an object
$X$ of $\mathcal{A}$ is called \emph{compact} when the functor
$\text{Hom}_\mathcal{A}(X,?):\mathcal{A}\longrightarrow\text{Ab}$
preserves coproducts. Recall that if $R$ is a ring, then the compact
objects of its derived category $\mathcal{D}(R)$ are the complexes
which are quasi-isomorphic to bounded complexes of finitely
generated projective modules (see \cite{R}).

Let $X$ and $V$ be objects of $\mathcal{G}$. We say that $X$ is
\emph{$V$-generated (resp. $V$-presented)} when there is an
epimorphism $V^{(I)}\twoheadrightarrow X$ (resp. an exact sequence
$V^{(J)}\longrightarrow V^{(I)}\longrightarrow X\rightarrow 0$), for
some sets $I$ and $J$.  We will denote by $\text{Gen}(V)$ and
$\text{Pres}(V)$  the classes of $V$-generated and $V$-presented
objects, respectively. The object $X$ always contains a largest
$V$-generated subobject, namely,
$tr_V(X)=\sum_{f\in\text{Hom}_\mathcal{G}(V,X)}\text{Im}(f)$. It is
called the \emph{trace of $V$ in $X$}. As a sort of dual concept,
given a class $\mathcal{S}$ of objects of $\mathcal{G}$, the
\emph{reject of $\mathcal{S}$ in $X$} is
$\text{Rej}_\mathcal{S}(X)=\bigcap_{f\in\text{Hom}_\mathcal{G}(X,S)}\text{Ker}(f)$.
We say that $X$ is \emph{$V$-subgenerated} when it is isomorphic to
a subobject of a $V$-generated object. The class of $V$-subgenerated
objects will be denoted by $\overline{\text{Gen}}(V)$. This
subcategory is itself a Grothendieck category and the inclusion
$\overline{\text{Gen}}(V)\hookrightarrow\mathcal{G}$ is an exact
functor. We will denote by $\text{Add}(V)$ (resp. $\text{add}(V)$)
the class of objects $X$ of $\mathcal{G}$ which are isomorphic to
direct summands of
 coproducts (resp. finite coproducts) of copies of $V$.

Given any category $\mathcal{C}$, an object $G$ is called a \emph{generator} of $\mathcal{C}$ when  the functor 
$\text{Hom}_\mathcal{C}(G,?):\mathcal{C}\longrightarrow\text{Sets}$ is faithful. When $\mathcal{C}=\mathcal{A}$ is cocomplete abelian, $G$ is a generator exactly when  $\text{Gen}(G)=\mathcal{A}$ (note that the definition of $\text{Gen}(V)$ is also valid in this context). 
 An object $G$ of $\mathcal{A}$ is called a
\emph{progenerator} when it is a compact projective generator. It is
a well-known result of Gabriel and Mitchell (see \cite[Corollary
3.6.4]{Po}) that $\mathcal{A}$ is  a module category if, and only
if, it has a progenerator. We will frequently use this
characterization of module categories in the paper.

 Slightly diverting from the terminology of \cite{CDT1} and
 \cite{CDT2}, an object $V$ of $\mathcal{G}$ will be called \emph{quasi-tilting} when
 $\text{Gen}(V)=\overline{\text{Gen}}(V)\cap\text{Ker}(\text{Ext}_\mathcal{G}^1(V,?))$.
 When, in addition, we have that
 $\overline{\text{Gen}}(V)=\mathcal{G}$, we will say that $V$ is
 a \emph{1-tilting object}. That is, $V$ is 1-tilting if, and only
 if, $\text{Gen}(V)=\text{Ker}(\text{Ext}_\mathcal{G}^1(V,?))$.
When $\mathcal{G}=R-\text{Mod}$, a module $V$ is 1-tilting if, and
only if, it satisfies the following three properties:
\hspace*{0.5cm} i) the projective dimension of $V$, denoted
$pd({}_RV)$, is $\leq 1$; \hspace*{0.5cm} ii)
$\text{Ext}_R^1(T,T^{(I)})=0$, for each set $I$; \hspace*{0.5cm}
iii) there exists and exact sequence $0\rightarrow R\longrightarrow
T^0\longrightarrow T^1\rightarrow 0$ in $R-\text{Mod}$, where
$T^i\in\text{Add}(T)$ for $i=0,1$ (see \cite[Proposition 1.3]{CT}).

 When
 $V$ is a quasi-tilting object of $\mathcal{G}$, we have that $\text{Gen}(V)=\text{Pres}(V)$
 and that $(\text{Gen}(V),\text{Ker}(Hom_\mathcal{A}(V,?)))$
  is a torsion pair in $\mathcal{G}$. In the particular case when $V$ is 1-tilting, this
  pair is called the  \emph{tilting torsion pair}
associated to $V$.  A \emph{classical quasi-tilting (resp. classical
1-tilting)} object is a quasi-tilting (resp. 1-tilting) object $V$
such that the canonical morphism
$\text{Hom}_\mathcal{G}(V,V)^{(I)}\longrightarrow\text{Hom}_\mathcal{G}(V,V^{(I)})$
is an isomorphism, for all sets $I$. By \cite[Proposition
2.1]{CDT1}, we know  that if $\mathcal{G}=R-\text{Mod}$, then a
classical quasi-tilting $R$-module is just a finitely generated
quasi-tilting module. Even more (see \cite[Proposition 1.3]{CT}), a
classical 1-tilting $R$-module is just a finitely presented
1-tilting $R$-module.

On what concerns triangulated categories, we will follow \cite{N}
and \cite{V} as basic texts, but if $\mathcal{D}$ is a triangulated
category, we will denote  by
$?[1]:\mathcal{D}\longrightarrow\mathcal{D}$ the suspension functor
and we will write triangles in the form $X\longrightarrow
Y\longrightarrow Z\stackrel{+}{\longrightarrow}$. A
\emph{triangulated functor} between triangulated categories is a
functor which preserves triangles. Given a triangulated category
$\mathcal{D}$, a \emph{t-structure} in $\mathcal{D}$ is a pair
$(\mathcal{U},\mathcal{W})$ of full subcategories, closed under
taking direct summands in $\mathcal{D}$, which satisfy the
 following  properties:

\begin{enumerate}
\item[i)] $\text{Hom}_\mathcal{D}(U,W[-1])=0$, for all
$U\in\mathcal{U}$ and $W\in\mathcal{W}$;
\item[ii)] $\mathcal{U}[1]\subseteq\mathcal{U}$;
\item[iii)] For each $X\in Ob(\mathcal{D})$, there is a triangle $U\longrightarrow X\longrightarrow
V\stackrel{+}{\longrightarrow}$ in $\mathcal{D}$, where
$U\in\mathcal{U}$ and $V\in\mathcal{W}[-1]$.
\end{enumerate}
It is easy to see that in such case $\mathcal{W}=\mathcal{U}^\perp
[1]$ and $\mathcal{U}={}^\perp (\mathcal{W}[-1])={}^\perp
(\mathcal{U}^\perp )$. For this reason, we will write a t-structure
as $(\mathcal{U},\mathcal{U}^\perp [1])$. The full subcategory
$\mathcal{H}=\mathcal{U}\cap\mathcal{W}=\mathcal{U}\cap\mathcal{U}^\perp
[1]$ is called the \emph{heart} of the t-structure and it is an
abelian category, where the short exact sequences 'are' the
triangles in $\mathcal{D}$ with their three terms in $\mathcal{H}$. In particular, one has $\text{Ext}^{1}_{\mathcal{H}}(M,N)=\text{Hom}_{\mathcal{D}}(M,N[1]),$ for all objects $M$ and $N$ in $\mathcal{H}$ (see \cite{BBD}).

 We will denote by $\mathcal{C}(\mathcal{G})$, $\mathcal{K}(\mathcal{G})$ and $\mathcal{D}(\mathcal{G)}$ the category of chain
 complexes of objects of $\mathcal{G}$, the homotopy category of $\mathcal{G}$ and the derived category of $\mathcal{G}$, respectively.
 In the particular case when $\mathcal{G}=R-\text{Mod}$, we will  write $\mathcal{C}(R):=\mathcal{C}(R-\text{Mod})$,
 $\mathcal{K}(R):=\mathcal{K}(R-\text{Mod})$ and
 $\mathcal{D}(R):=\mathcal{D}(R-\text{Mod})$.
Given a torsion pair $\mathbf{t}=(\mathcal{T},\mathcal{F})$ in
$\mathcal{G}$, extending to the unbounded context a construction due
to Happel-Reiten-Smal\o \ (see \cite{HRS}),  one gets a t-structure
$(\mathcal{U}_\mathbf{t},\mathcal{U}_\mathbf{t}^{\perp}[1])=(\mathcal{U}_\mathbf{t},\mathcal{W}_\mathbf{t})$
in $\mathcal{D}(\mathcal{G})$ by defining:

\begin{center}
$\mathcal{U}_\mathbf{t}=\{X\in\mathcal{D}^{\leq 0}(\mathcal{G}):$
$H^0(X)\in\mathcal{T}\}$

$\mathcal{W}_\mathbf{t}=\{Y\in\mathcal{D}^{\geq -1}(\mathcal{G}):$
$H^{-1}(Y)\in\mathcal{F}\}$.
\end{center}
In this case, the heart $\mathcal{H}_\mathbf{t}$ consists of the
complexes $M$ such that $H^{-1}(M)\in\mathcal{F}$,
$H^0(M)\in\mathcal{T}$ and $H^k(M)= 0$, for all $k\neq -1,0$. We
will say that $\mathcal{H}_\mathbf{t}$ is the \emph{heart of the
torsion pair $\mathbf{t}$}.

When $\mathcal{G}=R-\text{Mod}$, we will frequently deal  with complexes $\cdots
\longrightarrow 0\longrightarrow
X\stackrel{j}{\longrightarrow}Q\stackrel{d}{\longrightarrow}P\longrightarrow
0 \longrightarrow \cdots$, concentrated in degrees $-2,-1,0$, such
that $j$ is a monomorphism and $P,Q$ are projective modules. All throughout the paper such a complex will be said to be a \emph{complex in standard form} and, without loss of generality, we will assume that $X$ is a submodule of $Q$ and $j$ is the inclusion. If $\mathbf{t}$ is a torsion pair in $R-\text{Mod}$, then each object of $\mathcal{H}_\mathbf{t}$ is quasi-isomorphic to a complex in standard form. Moreover, 
if $M$
and $N$ are two  complexes in standard form and they represent objects of
$\mathcal{H}_\mathbf{t}$,  then the canonical map
$\text{Hom}_{\mathcal{K}(R)}(M,N)\longrightarrow\text{Hom}_{\mathcal{D}(R)}(M,N)=\text{Hom}_{\mathcal{H}_\mathbf{t}}(M,N)$
is bijective. We will frequently use this fact throughout the paper.

 An object $T$ of
a triangulated category $\mathcal{D}$ will be called \emph{classical
tilting} when satisfies the following  conditions: i) $T$ is compact
in $\mathcal{D}$; ii) $\text{Hom}_\mathcal{D}(T,T[i])=0$, for all
$i\neq 0$;  and iii) if $X\in\mathcal{D}$ is an object such that
$\text{Hom}_\mathcal{D}(T[i],X)=0$, for all $i\in\mathbb{Z}$, then
$X=0$.  For instance, if $T$ is a classical 1-tilting $R$-module,
then $T[0]$ is a classical tilting object of $\mathcal{D}(R)$. By a
well-known result of Rickard (see \cite{R} and \cite{R2}), two rings
$R$ and $S$ are \emph{derived equivalent}, i.e., have equivalent
derived categories, if and only if there exists a classical tilting
object $T$ in $\mathcal{D}(R)$ such that
$S\cong\text{End}_{\mathcal{D}(R)}(T)^{op}$.

Let $P^\bullet :$ $\cdots \longrightarrow0\longrightarrow
Q\stackrel{d}{\longrightarrow}P\longrightarrow 0 \longrightarrow
\cdots $ be a complex of finitely generated projective $R$-modules
concentrated in degrees $-1$ and $0$. In \cite{HKM}, the authors
associated to such a complex a pair $(\mathcal{X}(P^\bullet
),\mathcal{Y}(P^\bullet ))$ of full subcategories of $R-\text{Mod}$
defined as follows, where $M$ is an $R$-module:

\begin{center}
$M\in\mathcal{X}(P^\bullet )$ $\Longleftrightarrow$
$\text{Hom}_{\mathcal{D}(R)}(P^\bullet ,M[1])=0$

$M\in\mathcal{Y}(P^\bullet )$ $\Longleftrightarrow$
$\text{Hom}_{\mathcal{D}(R)}(P^\bullet ,M[0])=0$.
\end{center}
Under some precise conditions (see \cite[Theorem 2.10]{HKM}), the
pair $(\mathcal{X}(P^\bullet ),\mathcal{Y}(P^\bullet ))$ is a torsion
pair in $R-\text{Mod}$. When this is the case, we shall say that
$P^\bullet$ is an \emph{HKM complex} and that
$\mathbf{t}=(\mathcal{X}(P^\bullet ),\mathcal{Y}(P^\bullet ))$ is the
associated \emph{HKM torsion pair}.

For any ring $R$, we shall denote by $V(R)$ the additive monoid
whose elements are the isoclasses of finitely generated projective
$R$-modules, where $[P]+[Q]=[P\oplus Q]$. For each two-sided ideal
$\mathbf{a}$ of the ring $R$, we have an obvious morphism of monoids
$V(R)\longrightarrow V(R/\mathbf{a})$ taking $[P]\rightsquigarrow
[P/\mathbf{a}P]$. This morphism need not be surjective. However, the
class of rings $R$ for which it is surjective, independently of
$\mathbf{a}$,  is very large and includes the so-called exchange
rings (see \cite[Lemma 3.2, Theorem 3.3]{A}). This class of rings
includes all rings which are Von Neumann regular modulo the Jacobson
radical and which have the lifting of idempotents property with
respect to this radical. In particular, it includes all
\emph{semiperfect rings}, i.e., those rings $R$ such that $R/J(R)$
is semisimple and idempotents lift modulo $J(R)$, where $J(R)$
denotes the Jacobson radical of $R$.
  All local and all (left or
right) artinian rings, in particular all Artin algebras, are
semiperfect rings.

For concepts not explicitly defined in the paper, the reader is
referred to \cite{P} or \cite{Po} for those concerning arbitrary and
abelian categories, to \cite{K} and \cite{S} for those concerning
rings and their module categories and to  \cite{N} and \cite{V} for
those concerning triangulated categories.

\section{When is the heart of a torsion pair a module category?}
All throughout the paper, $R$ will be a ring and
$\mathbf{t}=(\mathcal{T},\mathcal{F})$ will be a torsion pair in
$R-\text{Mod}$. Unless otherwise stated, the letter $G$ will be denote a complex in standard form.  
 Frequently,  such a complex will satisfy some or all of the following conditions with respect to $\mathbf{t}$, to which we will refer as the  \emph{standard conditions} (here $V=H^{0}(G)$):

\begin{enumerate}
\item  $\mathcal{T}=Pres(V)\subseteq \text{Ker}(\text{Ext}_{R}^{1}(V,?))$;
\item  $Q$ and $P$ are finitely generated projective $R$-modules;

\item $H^{-1}(G)\in\mathcal{F}$ and $H^{-1}(G) \subseteq \text{Rej}_{\mathcal{T}}(\frac{Q}{X})$;
\item $\text{Ext}_R^1(\frac{Q}{X},?)$ vanishes on $\mathcal{F}$;
\item there is a morphism $\xymatrix{h:(\frac{Q}{X})^{(I)}\ar[r] & \frac{R}{t(R)}}$, for some set $I$, such that the cokernel of its
 restriction  to $(H^{-1}(G))^{(I)}$  is in  $\overline{Gen}(V)$.
\end{enumerate}

\begin{lemma} \label{lem.Hom and Ext in H}
Let $G$ be a complex in standard form and let $M$ be any $R$-module. The following assertions hold: 

\begin{enumerate}
\item There is an isomorphism $\xymatrix{\text{Hom}_R(H^0(G),M) \ar[r]^{\sim \hspace{0.1 cm}}& \text{Hom}_{\mathcal{D}(R)}(G,M[0])}$,
which is natural in $M$;

\item When we view $X$ as a submodule of $Q$ and $j$ as the inclusion, there are natural in $M$ exact sequences of abelian groups:

\begin{enumerate}
\item $\xymatrix{\text{Hom}_R(P,M) \ar[r] & \text{Hom}_R(Q/X,M) \ar[r] & \text{Hom}_{\mathcal{D}(R)}(G,M[1]) \ar[r] & 0}$.
\item $\xymatrix{\text{Hom}_R(Q,M) \ar[r] & \text{Hom}_R(X,M) \ar[r] & \text{Hom}_{\mathcal{D}(R)}(G,M[2])\ar[r] & 0}$.
\end{enumerate}
\end{enumerate}
\end{lemma}
\begin{proof}
We have triangles in $\mathcal{D}(R)$:

\begin{center}
$\xymatrix{H^{-1}(G)[1] \ar[r] &  G \ar[r] &
H^0(G)[0]\ar[r]^{\hspace{0.6 cm}+} &}$

and

$\xymatrix{Q/X[0]\ar[r] & P[0] \ar[r] & G\ar[r]^{+} &}$.
\end{center}
Applying the cohomological  functor
$\text{Hom}_{\mathcal{D}(R)}(?,M[0])$ and looking at the
corresponding long exact sequences, we obtain assertions 1 and 2.a.
On the other hand, one easily sees that a morphism
$\xymatrix{G\ar[r] & M[2]}$ in $\mathcal{D}(R)$ is represented by an
$R$-homomorphism $\xymatrix{f:X \ar[r] & M}$. The former morphism is
the zero morphism in $\mathcal{D}(R)$ precisely when $f$ factors
through $j$. Then the exact sequence in 2.b follows immediatly.
\end{proof}

\begin{lemma} \label{lem.description  de t}
If $G$ is a progenerator of $\mathcal{H}_{\mathbf{t}}$, then the
following assertions hold, where $V:=H^{0}(G)$:

\begin{enumerate}
\item $\mathcal{T}=Gen(V)=Pres(V)$, and hence  $\mathcal{F}=Ker(Hom_{R}(V,?))$;

\item $V$ is a finitely presented $R$-module;

\item $V$ is a classical quasi-tilting $R$-module.
\end{enumerate}
\end{lemma}

\begin{proof}
By hypothesis $\mathcal{H}_{\mathbf{t}}$ is a module category, in
particular $\mathcal{H}_{\mathbf{t}}$ is an AB5 category, so that
$\mathcal{F}$ is closed under taking direct limits in $R\text{-Mod}$
(see \cite[Theorem 4.8]{PS}). On the other hand, by \cite[Lemma
4.1]{PS}, the functor $\xymatrix{H^0:\mathcal{H}_\mathbf{t} \ar[r] &
R\text{-Mod}}$ is right exact and preserves coproducts. When applied
to an exact sequence $\xymatrix{G^{(I)} \ar[r] & G^{(J)} \ar[r] &
T[0] \ar[r] & 0}$ in $\mathcal{H}_\mathbf{t}$,  we get that
$T\in\text{Pres}(V)$, for each $T\in\mathcal{T}$. We then get that
$\mathcal{T}=Pres(V)$, and assertion 1  follows from
\cite[Proposition 2.2]{MT}.

Without loss of generality we can assume that G is in standard form. If $(T_i)_{i \in I}$ is a direct system
in $\mathcal{T}$, then
$\varinjlim_{\mathcal{H}_{\mathbf{t}}}{(T_i[0])}\cong
(\varinjlim{T_i})[0]$ (see \cite[Proposition 4.2]{PS}). We then get
that $Hom_{R}(V,?)$ preserves direct limits of objects in
$\mathcal{T}$  since $G$ is a compact object of
$\mathcal{H}_{\mathbf{t}}$.
 Let now $(M_i)_{i\in I}$ be any direct system in
R-Mod. We then get that $\varinjlim{t(M_i)}\cong t(\varinjlim
{M_i})$ since $\varinjlim\mathcal{F}=\mathcal{F}$. We now have
isomorphisms

{\small
\begin{center}
$\xymatrix{\varinjlim\text{Hom}_R(V,M_i) &
\varinjlim\text{Hom}_R(V,t(M_i)) \ar[r]^{\sim \hspace{1.7 cm}}
\ar[l]_{\sim} & \text{Hom}_R(V,\varinjlim
t(M_i))=\text{Hom}_R(V,t(\varinjlim M_i))
\ar[r]^{\hspace{1.8cm}\sim}&\text{Hom}_R(V,\varinjlim M_i)}$.
\end{center}}
Then assertion 2 follows. Finally, assertion 3
 follows  from \cite[Proposition 2.4]{MT}, from assertions 1 and
 2 and from \cite[Proposition 2.1]{CDT1}.

\end{proof}

The following result is inspired by \cite[Proposition 5.9]{CMT}.
\begin{lemma}\label{lem.generate (1:t)R}
Let $G$ be a complex in standard form. If $G$ is a projective object of $\mathcal{H}_{\mathbf{t}}$ such that
$\mathcal{T}=Gen(V)=Pres(V)$, where $V:=H^{0}(G)$, then the
following assertions hold:
\begin{enumerate}
\item $\frac{M}{t(M)}[1]\in Gen_{\mathcal{H}_{\mathbf{t}}}(G)$, for each $M\in \overline{Gen}(V)$;

\item $\frac{R}{t(R)}[1]\in Gen_{\mathcal{H}_{\mathbf{t}}}(G)$ if, and only if, $G$ satisfies the standard condition 5.
\end{enumerate}
\end{lemma}
\begin{proof}
With an easy adaptation, assertion 1 follows from \cite[Lemma 5.8 and Proposition 5.9]{CMT}.

We now prove assertion 2. For the \emph{only if} part, by hypothesis, there are a set $I$ and an exact sequence 
$$\xymatrix{0 \ar[r] & K \ar[r] & G^{(I)} \ar@{>>}^{p}[r] & \frac{R}{t(R)}[1] \ar[r] & 0}$$
in $\mathcal{H}_{\mathbf{t}}$. It is easy to see that $p$ is represented by an $R$-homomorphism $p^{-1}:(\frac{Q}{X})^{(I)} \xymatrix{\ar[r] & } \frac{R}{t(R)}$. Moreover, we have that $H^{-1}(p)$ coincides with the restriction of $p^{-1}$ to $(H^{-1}(G))^{(I)}$. Now, we consider the long exact sequence associated to the above triangle:

$$\xymatrix{0 \ar[r] & H^{-1}(K) \ar[r] & H^{-1}(G)^{(I)} \ar[rr]^{H^{-1}(p)} && \frac{R}{t(R)} \ar[r] & H^{0}(K) \ar[r] & V^{(I)} \ar[r] & 0}$$
The result follows by putting $h=p^{-1}$ since $H^{0}(K)\in \mathcal{T}$. 

For the \emph{if part},  assume that the standard condition 5 holds and let $Z$ be the cokernel of the  restriction of $h$ to
$(H^{-1}(G))^{(I)}$.
 Clearly, we can extend $h$ to a morphism from
$G^{(I)}$ to $\frac{R}{t(R)}[1]$ in $D(R)$, which  we denote by
$\bar{h}$. We now complete $\bar{h}$ to a triangle in $D(R)$, we
get:
$$\xymatrix{M \ar[r] & G^{(I)} \ar[r]^{\bar{h} \hspace{0.2 cm}} & \frac{R}{t(R)}[1] \ar[r]^{\hspace{0.4cm}+} &  & }$$

Using the long exact sequence of homologies, we then obtain an exact
sequence in $R$-Mod of the form:

$$\xymatrix{0 \ar[r] & Z \ar[r] & H^{0}(M) \ar[r] & V^{(\alpha)} \ar[r]& 0}$$

By \cite[Lemma 5.6]{CMT}, we get that $H^{0}(M) \in
\overline{Gen}(V)$ and then, by assertion 1, we also get that
$\frac{H^{0}(M)}{t(H^{0}(M))}[1] \in
Gen_{\mathcal{H}_{\mathbf{t}}}(G)$. Consider now the following
diagram commutative
\begin{small}
$$\xymatrix{&&&& &\\&&& t(H^{0}(M))[1]\ar[dd] \ar[ru]^{+} & \\ & N \ar[rru] \ar[rd] & && \\ H^{-1}(M)[2] \ar[rr] \ar[ru] & &M[1] \ar[r] \ar[rd]& H^{0}(M)[1] \ar[r]^{\hspace{0.6 cm}+} \ar[d] & \\ &&& \frac{H^{0}(M)}{t(H^{0}(M))}[1] \ar[rd]^{\hspace{0.2 cm} +}& \\ &&&&   }$$
\end{small}
Note that $N\in \mathcal{H}_{\mathbf{t}}[1]$. By \cite{BBD}, we get that 
$Coker_{\mathcal{H}_{\mathbf{t}}}(\bar{h})\cong
\frac{H^{0}(M)}{t(H^{0}(M))}[1]$. Then we have the following diagram
with exact row in $\mathcal{H}_{\mathbf{t}}$:

$$\xymatrix{&& G^{(J)} \ar@{>>}[d]^{p} \ar@{-->}[ld]_{p'} &\\G^{(I)} \ar[r]^{\bar{h}\hspace{0.2cm}} & \frac{R}{t(R)}[1] \ar[r] &
\frac{H^{0}(M)}{t(H^{0}(M))}[1] \ar[r] & 0}$$

where $p$ is an epimorphism  and $p'$ is obtained by the
projectivity of $G^{(J)}$ in $\mathcal{H}_\mathbf{t}$. It follows
that $\xymatrix{(\bar{h} \hspace{0.3cm} p'):G^{(I)}\coprod
G^{(J)}\ar[r] & \frac{R}{t(R)}[1]}$ is also an epimorphism in
$\mathcal{H}_{\mathbf{t}}$.

\end{proof}

We are now able to give a general criterion for
$\mathcal{H}_\mathbf{t}$ to be a module category.

\begin{teor}\label{prop.module heart of torsion pair}
Let $R$ be a ring and $\mathbf{t}$ be a torsion pair in $R-\text{Mod}$. A complex $G$  is a progenerator of the heart $\mathcal{H}_\mathbf{t}$  if, and only
if, it is quasi-isomorphic to a complex in standard form satisfying the  standard conditions 1-5. 
In particular $\mathcal{H}_\mathbf{t}$ is a module category if, and only if, this latter complex exists. 

\end{teor}
\begin{proof}
Let us assume that $G$ is  a complex in standard form which is in $\mathcal{H}_\mathbf{t}$. By lemma
\ref{lem.description  de t}, if $G$ is a progenerator of
$\mathcal{H}_\mathbf{t}$, then $V:=H^0(G)$ is finitely presented.
This allows us, for both implications in the proof,  to assume that
$P$ is a finitely generated projective $R$-module.

We claim that $G$ is a projective object in $\mathcal{H}_\mathbf{t}$
if, and only if, $\mathcal{T} \subseteq
\text{Ker}(\text{Ext}_{R}^{1}(V,?))$ and the standard conditions 3 and 4 hold.
Indeed each object $M\in\mathcal{H}_\mathbf{t}$ fits into an exact
sequence in this category

\begin{center}
$\xymatrix{0\ar[r] & H^{-1}(M)[1] \ar[r] & M \ar[r] & H^0(M)[0]
\ar[r] &  0 & (*)}$
\end{center}
Then $G$ is projective in $\mathcal{H}_\mathbf{t}$ if, and only if,
$0=\text{Ext}_\mathcal{H_\mathbf{t}}^1(G,T[0])=\text{Hom}_{\mathcal{D}(R)}(G,T[1])$
and
$0=\text{Ext}_\mathcal{H_\mathbf{t}}^1(G,F[1])=\text{Hom}_{\mathcal{D}(R)}(G,F[2])$,
for all $T\in\mathcal{T}$ and $F\in\mathcal{F}$. By lemma
\ref{lem.Hom and Ext in H}, the first equality holds if, and only
if, the  map
$\xymatrix{\text{Hom}_R(P,T)\ar[r]^{\bar{d}^{*}\hspace{0.2 cm}} &
\text{Hom}_R(Q/X,T)}$ is surjective, where
$\xymatrix{\bar{d}:Q/X\ar[r] & P}$ is the obvious $R$-homomorphism.
But, in turn, this last condition is equivalent to the sum of the
following two conditions, for each $T\in\mathcal{T}$:

\begin{enumerate}
\item[i)] Each $R$-homomorphism $\xymatrix{f:Q/X\ar[r] & T}$ vanishes on $H^{-1}(G)=\frac{\text{Ker}(d)}{X}$;
\item[ii)] Each morphism $\xymatrix{g:\text{Im}(\bar{d})=\text{Im}(d)\ar[r] &
T}$ extends to $P$.
\end{enumerate}
Condition i) is equivalent to the standard condition 3. On the other hand, condition ii) above is equivalent to
saying that $\text{Ext}_R^1(H^0(G),T)=0$, for all $T\in\mathcal{T}$.
Now, by lemma \ref{lem.Hom and Ext in H}, the equality
$\text{Hom}_{\mathcal{D}(R)}(G,F[2])=0$ holds when each
$R$-homomorphism $\xymatrix{g:X\ar[r] &F}$ extends to $Q$, for all
$F\in\mathcal{F}$. This is clearly equivalent to the standard condition 4.

Suppose that $G$ is projective in $\mathcal{H}_\mathbf{t}$ or its
equivalent conditions mentioned in the previous paragraph. Recall
from \cite[Section 4]{PS} that $\mathcal{H}_\mathbf{t}$ is AB4.
Applying this fact to any family of exact sequences as ($\ast$),  we
see that $G$ is a compact object of $\mathcal{H}_\mathbf{t}$ if, and
only if, the canonical morphisms

\begin{center}
$\xymatrix{\underset{i\in
I}{\coprod}\text{Hom}_{\mathcal{D}(R)}(G,T_i[0])\ar[r]
&\text{Hom}_{\mathcal{D}(R)}(G,(\underset{i\in I}{\coprod}
T_i)[0])}$

and

$\xymatrix{\underset{i\in I}{\coprod}
\text{Hom}_{\mathcal{D}(R)}(G,F_i[1])\ar[r]
&\text{Hom}_{\mathcal{D}(R)}(G,(\underset{i \in I}{\coprod}
F_i)[1])}$
\end{center}
are isomorphisms, for all families $(T_i)$ in $\mathcal{T}$ and
$(F_i)$ in $\mathcal{F}$. By lemma \ref{lem.Hom and Ext in H}, we
easily get that the first of these morphisms is an isomorphism
precisely when $V$ is a compact object of $\mathcal{T}$. On the
other hand, by lemma \ref{lem.Hom and Ext in H}(2),  the second
centered homomorphism is an isomorphism whenever $P$ and $Q/X$ are
finitely generated modules.  Therefore, if $G$ satisfies the standard 
conditions 1-4, then $G$ is a compact projective object
of $\mathcal{H}_{\mathbf{t}}$.

Suppose now that these last conditions hold. Then, due to the
canonical sequence $(\ast)$, we know that $G$ is a generator if, and
only if, each $M\in\mathcal{T}[0]\cup\mathcal{F}[1]$ is generated by
$G$. Note that we have an epimorphism $\xymatrix{q:G\ar@{>>}[r] &
V[0]}$ in $\mathcal{H}_\mathbf{t}$, which implies that
$\text{Gen}_{\mathcal{H}_\mathbf{t}}(V[0])\subseteq\text{Gen}_{\mathcal{H}_\mathbf{t}}(G)$.
But the equality $\mathcal{T}=\text{Gen}(V)=\text{Pres}(V)$ easily
gives that
$\mathcal{T}[0]\subseteq\text{Gen}_{\mathcal{H}_\mathbf{t}}(V[0])$.
On the other hand, each $F\in\mathcal{F}$ gives rise to an exact
sequence $\xymatrix{0 \ar[r] & F' \hspace{0.02cm}\ar@{^(->}[r] &
(\frac{R}{t(R)})^{(I)} \ar@{>>}[r] & F \ar[r] & 0}$ in
$R\text{-Mod}$ which, in turn, yields an exact sequence in
$\mathcal{H}_\mathbf{t}$:

\begin{center}
$\xymatrix{0 \ar[r] &
 F'[1] \ar[r] & (\frac{R}{t(R)})^{(I)}[1]\ar[r] &
F[1]\ar[r] & 0}$.
\end{center}
Thus, $G$ generates $\mathcal{H}_\mathbf{t}$ if, and only if, it
generates  $\frac{R}{t(R)}[1]$. By lemma \ref{lem.generate (1:t)R},
this is equivalent to the standard condition 5.

Note that the \emph{'if' part} of the proof follows from the
previous paragraphs. By  lemma \ref{lem.description  de t} and lemma
\ref{lem.generate (1:t)R}, in order to prove the \emph{'only if'
part}, we only need to prove that if  $G$ is a complex in standard form, with $P$ finitely generated,  and it is  a
progenerator of $\mathcal{H}_\mathbf{t}$, then $G$ is quasi-isomorphic to  a complex  satisfying the standard conditions 1-5. But Lemma \ref{lem.Q/X is finitely generated new} below shows that
 $Q/X$ is finitely generated, which allows us to replace $Q$ by a finitely generated projective module $Q'$ and get a complex

\begin{center}
$\xymatrix{ \cdots \ar[r] &  0 \ar[r] & X' \ar[r]&
Q'\ar[r]^{d'}& P \ar[r] & 0 \ar[r] &
 \cdots} $
\end{center}
which is quasi-isomorphic to $G$ and satisfies all the standard conditions.

\end{proof}

\begin{lemma} \label{lem.Q/X is finitely generated new}
Let $G$ be a complex in standard form with $P$ finitely generated. Suppose that $G$ is a progenerator of
$\mathcal{H}_\mathbf{t}$. Then $\frac{Q}{X}$ is a finitely
generated $R$-module.
\end{lemma}
\begin{proof}
We identify $G$ with the
complex
$$\xymatrix{\cdots \ar[r] &  0 \ar[r] & Q/X \ar[r]^{\bar{d}} & P \ar[r] & 0
\ar[r] & \cdots}$$ By lemma \ref{lem.description de t}, we know that
$V:=H^{0}(G)$ is a finitely presented $R$-module, thus
$Im(d)=Im(\bar{d})$ is a finitely generated submodule of $P$, we can
select a finitely generated submodule $A'<\frac{Q}{X}$ such that
$\bar{d}(A')=\text{Im}(d)$. We fix a direct system
$(A_{\lambda})_{\lambda \in \Lambda}$ of  finitely generated
submodules of $\frac{Q}{X}$, such that $A' < A_\lambda$ and
$\varinjlim{A_{\lambda}}=\frac{Q}{X}$. For each $\lambda\in\Lambda$, we denote by 
$G_{\lambda}$ the following complex:
$$\xymatrix{G_{\lambda}: \hspace{1 cm} \cdots \ar[r] & 0 \ar[r] &  A_{\lambda} \ar[r]^{\bar{d}_{|A_\lambda}} & P \ar[r] & 0 \ar[r] & \cdots}$$

It is clear that $(G_{\lambda})_{\lambda \in \Lambda}$ is a direct
system in $\mathcal{C}(R)$ and in $\mathcal{H}_{\mathbf{t}}$, and
that we have $\varinjlim_{\mathcal{C}(R)}{G_{\lambda}}\cong G$. By
lemma \ref{lem.description de t} and \cite[Lemma 4.4]{PS}, we have
that
$\varinjlim_\mathcal{H_\mathbf{t}}G_\lambda\cong\varinjlim_{\mathcal{C}(R)}G_\lambda
=G$. But, since $G$ is a finitely presented object of
$\mathcal{H}_\mathbf{t}$, the identify map $1_G:G\longrightarrow G$
factors in  this category in the form
$G\stackrel{f}{\longrightarrow}G_\mu\stackrel{i_\mu}{\longrightarrow}G$,
for some $\mu\in\Lambda$. It follows that $H^{-1}(\iota_\mu )$ is an
epimorphism and, therefore, it is an isomorphism. We then get a
commutative diagram with exact rows:

$$\xymatrix{0 \ar[r] & H^{-1}(G_{\mu}) \ar[r] \ar[d]^{\wr}_{H^{-1}(i_\mu)} & A_{\mu} \ar[r] \ar@{^{(}->}[d]^{i_\mu^{-1}} & \bar{d}(A_{\mu}) \ar[r] \ar@{=}[d]  & 0 \\ 0 \ar[r] & H^{-1}(G) \ar[r] & \frac{Q}{X} \ar[r]^{\bar{d} \hspace{0.3 cm}} & Im(d) \ar[r] & 0}$$

Therefore, $A_{\mu}\cong \frac{Q}{X}$ is a finitely generated
$R$-module.
\end{proof}

\begin{rem} \label{rem.completion to main theorem}
If $G$ is a complex in standard form satistying the standard condition 2 (i.e. $P$ and $Q$ are finitely generated), the proof of theorem \ref{prop.module heart of torsion pair} shows that $G$ is a progenerator of $\mathcal{H}_\mathbf{t}$ if, and only if, $G$ itself satisfies all standard conditions 1-5.
\end{rem}

Our next result in this section gives a criterion for a torsion pair
to be HKM:

\begin{prop} \label{prop.HKM}
Let $\xymatrix{P^{\bullet}:= \cdots \ar[r] & 0 \ar[r] & Q \ar[r]^{d}
& P \ar[r] & 0 \ar[r] & \cdots }$ be a complex of finitely generated
projective modules concentrated in degrees $-1$ and $0$, put 
$V=H^0(P^\bullet )$,  and let $\mathbf{t}=(\mathcal{T},\mathcal{F})$
be a torsion pair in $R\text{-Mod}$. The following assertions are
equivalent:

\begin{enumerate}
\item $P^\bullet$ is an HKM complex such that $\mathbf{t}$ is its associated HKM torsion pair;
\item $V\in\mathcal{X}(P^\bullet )$ and the complex
$\xymatrix{G:= \cdots \ar[r] & 0 \ar[r] & t(\text{Ker}(d))
\hspace{0.02cm}\ar@{^(->}[r]^{\hspace{0.6 cm}j} & Q
\ar[r]^{d\hspace{0.05cm}} & P \ar[r] & 0 \ar[r] & \cdots }$,
concentrated in degrees $-2,-1,0$, is a progenerator of
$\mathcal{H}_\mathbf{t}$;
\item $P^{\bullet}$ satisfies  the standard conditions 1 and 5, $\text{Ker}(d)\subseteq\text{Rej}_\mathcal{T}(Q)$ and  $\mathcal{X}(P^\cdot )\subseteq\mathcal{T}$  .
\end{enumerate}
\end{prop}
\begin{proof}
Note that the standard conditions 2 and 4 are automatically satisfied by $P^{\bullet}$. 
Note also that we have an exact sequence $\xymatrix{0 \ar[r] &
t(\text{Ker}(d))[1]\ar[r] & P^\bullet \ar[r] & G'\ar[r] & 0}$ in
$\mathcal{C}(R)$, where $G'$ is quasi-isomorphic to $G$. We then get
a triangle in $\mathcal{D}(R)$:

\begin{center}
$\xymatrix{t(\text{Ker}(d))[1]\ar[r] & P^\bullet \ar[r] &
G\ar[r]^{+} & }$
\end{center}
In particular, we get a natural  isomorphism
$\xymatrix{\text{Hom}_{\mathcal{D}(R)}(G,?[0])\ar[r]^{\sim
\hspace{0.2cm}} &\text{Hom}_{\mathcal{D}(R)}(P^\bullet,?[0])}$ and
an exact sequence of functors $\xymatrix{R\text{-Mod}\ar[r] & Ab}$:

\begin{small}
$$\xymatrix{0 \ar[r] & \text{Hom}_{\mathcal{D}(R)}(G,?[1]) \ar[r] & \text{Hom}_{\mathcal{D}(R)}(P^\bullet,?[1]) \ar[r] & \text{Hom}_{R}(t(\text{Ker}(d)),?)
\ar[r] &\text{Hom}_{\mathcal{D}(R)}(G,?[2]) \ar[r] & 0}$$
\end{small}

 $1)\Longrightarrow 2)$
 We have an isomorphism $\xymatrix{\text{Hom}_{\mathcal{H}_\mathbf{t}}(G,?)=\text{Hom}_{\mathcal{D}(R)}(G,?)_{|\mathcal{H}_\mathbf{t}} \ar[r]^{\hspace{1cm}\sim} & \text{Hom}_{\mathcal{D}(R)}(P^\bullet,?)_{|\mathcal{H}_\mathbf{t}}}$ of functors
$\xymatrix{\mathcal{H}_\mathbf{t}\ar[r] & \text{Ab}}$ since
$\text{Hom}_{\mathcal{D}(R)}(t(\text{Ker}(d))[k],?)$ vanishes on
$\mathcal{H}_\mathbf{t}$, for $k=1,2$. It follows that the functor
$\xymatrix{\text{Hom}_{\mathcal{H}_\mathbf{t}}(G,?):\mathcal{H}_\mathbf{t}\ar[r]
&\text{Ab}}$ is faithful since the induced functor
\begin{small}$\xymatrix{\text{Hom}_{\mathcal{D}(R)}(P^\bullet
,?)_{|\mathcal{H}_\mathbf{t}}:\mathcal{H}_\mathbf{t} \ar[r] &
\text{End}_{\mathcal{D}(R)}(P^\bullet)\text{-Mod}}$\end{small} is an
equivalence of categories (see \cite[Theorem 2.15]{HKM}). Then $G$
is a generator of $\mathcal{H}_\mathbf{t}$. But the functor
$\xymatrix{\text{Hom}_{\mathcal{H}_\mathbf{t}}(G,?)
\cong\text{Hom}_{\mathcal{D}(R)}(P^\bullet,?)_{|\mathcal{H}_\mathbf{t}}:\mathcal{H}_\mathbf{t}
\ar[r] & \text{Ab}}$ preserves coproducts since $P^{\bullet}$ is a
compact object of $\mathcal{D}(R)$. It follows that $G$ is a compact
object of $\mathcal{H}_\mathbf{t}$.

From the initial comments of this proof and the fact that
$\mathcal{T}=\mathcal{X}(P^\bullet
)=\mathcal{}\text{Ker}(\text{Hom}_{\mathcal{D}(R)}(P^\bullet
,?[1])_{| R\text{-Mod}})$, we get a monomorphism
$\xymatrix{\text{Ext}_{\mathcal{H}_\mathbf{t}}^1(G,T[0])=\text{Hom}_{\mathcal{D}(R)}(G,T[1])\hspace{0.02
cm}\ar@{^(->}[r] &\text{Hom}_{\mathcal{D}(R)}(P^\bullet ,T[1])=0}$,
for each $T\in\mathcal{T}$. On the other hand,
$\text{Ext}_{\mathcal{H}_\mathbf{t}}^1(G,F[1])=\text{Hom}_{\mathcal{D}(R)}(G,F[2])$
is a homomorphic image of $\text{Hom}_R(t(\text{Ker}(d)),F)=0$, for
all $F\in\mathcal{F}$. We  conclude that $G$ is a projective object,
and hence a progenerator, of $\mathcal{H}_\mathbf{t}$ since
$\text{Ext}_{\mathcal{H}_\mathbf{t}}^1(G,?)$ vanishes on
$\mathcal{T}[0]$ and on $\mathcal{F}[1]$.

$2)\Longrightarrow 1)$ The mentioned initial comments show that
$\mathcal{Y}(P^\cdot )$ consists of the modules $F$ such that
$\text{Hom}_R(H^0(P^\bullet ),F)\cong\text{Hom}_{\mathcal{D}(R)}(G,
F[0])=0$. But theorem \ref{prop.module heart of torsion pair}
and its proof tell us that $H^0(G)=V$ generates $\mathcal{T}$, so
that we have $\mathcal{Y}(P^\bullet )=\mathcal{F}$. On the other
hand, if $F\in\mathcal{X}(P^\bullet )\cap\mathcal{Y}(P^\bullet
)=\mathcal{X}(P^\bullet )\cap\mathcal{F}$ then, again, our initial
comments in this proof show that
$\text{Hom}_{\mathcal{D}(R)}(G,F[1])=0$. But this implies that $F=0$
since $G$ is a generator of $\mathcal{H}_\mathbf{t}$. Assertion 1
follows now from \cite[Theorem 2.10]{HKM} and the fact that
$\mathcal{Y}(P^\bullet)=\mathcal{F}$.

$1),2)\Longrightarrow 3)$ From theorem \ref{prop.module heart
of torsion pair} and remark \ref{rem.completion to main theorem} we know that the complex $G$ satisfies all the standard conditions. 
It immediately follows that $P^\bullet$ satisfies the standard condition 1. 
As for standard condition 5,  note that we have isomorphisms of functors:

\begin{center}
$\xymatrix{\text{Hom}_R(Q^{(J)},?)_{|\mathcal{F}}\ar[r]^{\sim
\hspace{0.65
cm}}&\text{Hom}_R((Q/X)^{(J)},?)_{|\mathcal{F}}}$

$\xymatrix{\text{Hom}_R(\text{Ker}(d)^{(J)},?)_{|\mathcal{F}}\ar[r]^{\sim
\hspace{0.25 cm}} &\text{Hom}_R(H^{-1}(G)^{(J)},?)_{|\mathcal{F}}}$,
\end{center}
where $X=t(\text{Ker}(d))$.
Then the standard condition 5 holds for $P^{\bullet}$ because it holds for $G$. Finally, any
homomorphism $\xymatrix{f:Q\ar[r] &T}$, with $T\in\mathcal{T}$,
gives a morphism $\xymatrix{P^\bullet \ar[r] & T[1]}$ in
$\mathcal{D}(R)$. But this is the zero morphism since
$\mathcal{T}=\mathcal{X}(P^\bullet )$. This implies that $f$ factors
through $d$, so that $f(\text{Ker}(d))=0$ and, hence, that $\text{Ker}(d)\subseteq\text{Rej}_\mathcal{T}(Q)$.

 $3)\Longrightarrow 1)$ By lemma \ref{lem.Hom and Ext in H}, we know
 that $\mathcal{Y}(P^\bullet)$ consists of the modules $Y$ such that
 $ \text{Hom}_R(V,Y)=0$. Standard condition 1 gives then that $\mathcal{Y}(P^\bullet )=\mathcal{F}$, which implies that  $\mathcal{X}(P^\bullet )\cap\mathcal{Y}(P^\bullet )=0$  since $\mathcal{X}(P^{\bullet})\subseteq\mathcal{T}$.   On the other hand, the standard condition 3 says that each homomorphism $\xymatrix{f:Q\ar[r] & V}$
vanishes on $\text{Ker}(d)$. It then induces an $R$-homomorphism
$\xymatrix{\bar{f}:\text{Im}(d)\ar[r]& V}$, which
necessarily extends to $P$ since $\text{Ext}_R^1(V,V)=0$. This
proves that $\text{Hom}_{\mathcal{D}(R)}(P^\bullet ,V[1])=0$, thus
showing that $H^0(P^\bullet )=V\in\mathcal{X}(P^\bullet
 )$. Then, by \cite[Theorem 2.10]{HKM}, the pair $(\mathcal{X}(P^\bullet ),\mathcal{Y}(P^\bullet ))$
 is a torsion pair, which is necessarily equal to $\mathbf{t}$.
\end{proof}

\begin{cor} \label{cor.classical tilting progenerator}
Let $\mathbf{t}=(\mathcal{T},\mathcal{F})$ be a torsion pair in
$R\text{-Mod}$ and let $\xymatrix{P^\bullet:= \cdots \ar[r] & 0
\ar[r] & Q\ar[r]^{d} & P\ar[r] & 0 \ar[r] & \cdots}$ be a complex of
finitely generated projective $R$-modules concentrated in degrees
$-1$ and $0$. The following assertions are equivalent:

\begin{enumerate}

\item $P^\bullet$ is a classical tilting complex and $\mathbf{t}$ is the 
associated HKM torsion pair.
\item $P^\bullet$ is a progenerator of $\mathcal{H}_\mathbf{t}$;
\item The complex $P^{\bullet}$ satisfies the standard conditions (1, 3 and 5).
\end{enumerate}
\end{cor}
\begin{proof}

$1)\Longrightarrow 2)$ It follows directly from \cite[Remark 3.9 
and Theorem 3.8]{HKM}.

$2)\Longrightarrow 1)$ Let $M$ be an $R$-module and let us apply the 
cohomological functor $\text{Hom}_{D(R)}(P^\bullet ,?)$ to the canonical 
triangle 
$$\xymatrix{t(M)[0] \ar[r] &  M[0] \ar[r] & 
M/t(M)[0] \ar[r]^{\hspace{0.8 cm}+} & }$$
Using the fact that $P^\bullet$ 
is a progenerator of $\mathcal{H}_\mathbf{t}$, that 
$\text{Hom}_{D(R)}(P^\bullet ,?[0])$ vanish on $\mathcal{F}$ and that 
$\text{Hom}_{D(R)}(P^\bullet ,?[1])$ vanish on $\mathcal{T}$, we get that

\begin{center}
$M\in\mathcal{X}(P^\bullet )$ $\Longleftrightarrow$ 
$\text{Hom}_{D(R)}(P^\bullet ,M/t(M)[1])=0$ $\Longleftrightarrow$ 
$M/t(M)=0$ $\Longleftrightarrow$ $M\in\mathcal{T}$,
\end{center}

and also that

\begin{center}
$M\in\mathcal{Y}(P^\bullet )$ $\Longleftrightarrow$ 
$\text{Hom}_{D(R)}(P^\bullet ,t(M)[0])=0$ $\Longleftrightarrow$ $t(M)=0$ 
$\Longleftrightarrow$ $M\in\mathcal{F}$.
\end{center}
We then have that $\mathbf{t}=(\mathcal{X}(P^{\bullet }),\mathcal{Y}(P^{\bullet}))$  and, 
by \cite[Remark 3.9]{HKM}, the complex $P^\bullet$ is classical tilting.


$2)\Longleftrightarrow 3)$ is a direct consequence of theorem
\ref{prop.module heart of torsion pair} (see remark \ref{rem.completion to main theorem}).

\end{proof}

\begin{defi}
We shall say that $\mathcal{H}_\mathbf{t}$ \emph{has a progenerator
which is a classical tilting complex} when it has a progenerator
$P^\bullet$ as in corollary \ref{cor.classical tilting
progenerator}.
\end{defi}

\section{The case of a hereditary torsion pair}

Suppose now that $\mathbf{t}=(\mathcal{T},\mathcal{F})$ is
hereditary. We will show that the condition that its heart be a
module category gives more precise information than in the general
case. Recall that $\mathbf{t}$ is called \emph{bounded} when its
associated Gabriel topology has a basis consisting of two-sided
ideals (see \cite[Chapter VI]{S}). Equivalently, when
$R/\text{ann}_R(T)\in\mathcal{T}$, for each finitely generated module $T\in\mathcal{T}$. 

\begin{teor} \label{teor.hereditary case}
Let $\mathbf{t}=(\mathcal{T},\mathcal{F})$ be a hereditary torsion
pair in $R-\text{Mod}$ and let $G$ be a complex in the standard form,
where $P$ and $Q$ are finitely generated projective modules and $V=H^0(G)$. The following
assertions are equivalent:

\begin{enumerate}
\item $G$ is a progenerator of $\mathcal{H}_\mathbf{t}$;
\item The following conditions are satisfied:

\begin{enumerate}
\item $G$ satisfies the standard conditions 1, 3 and 4;
\item If $\mathbf{b}=\text{ann}_R(V/t(R)V)$, then $\mathbf{b}/t(R)$
is an idempotent ideal of $R/t(R)$ (which is finitely generated on
the left) and $R/\mathbf{b}$ is in $\mathcal{T}$;
\item There is a morphism
$\xymatrix{h:(\frac{Q}{X})^{(J)}\ar[r] &\frac{\mathbf{b}}{t(R)}}$,
for some set $J$,  such that
$\xymatrix{h_{|H^{-1}(G)^{(J)}}:H^{-1}(G)^{(J)}\ar[r]
&\frac{\mathbf{b}}{t(R)}}$ is an epimorphism.
\end{enumerate}
When $\mathbf{t}$ is  bounded, the assertions are also equivalent
to:

\item There is an idempotent ideal $\mathbf{a}$ of $R$, which is
finitely generated on the left, such that:

\begin{enumerate}
\item $\text{add}(V)=\text{add}(R/\mathbf{a})$ and
$\mathbf{t}$ is the right constituent torsion pair of the TTF triple
defined by $\mathbf{a}$;
\item $\text{Ker}(d)\subseteq X+\mathbf{a}Q$;
\item $\text{Ext}_R^1(\frac{Q}{X},?)$ vanishes on $\mathcal{F}$;
\item There is a morphism
$\xymatrix{h:(\frac{Q}{X})^{(J)}\ar[r]
&\frac{\mathbf{a}}{t(\mathbf{a})}}$, for some set $J$,  such that
$\xymatrix{h_{|H^{-1}(G)^{(J)}}:H^{-1}(G)^{(J)}\ar[r] &
\frac{\mathbf{a}}{t(\mathbf{a})}}$ is an epimorphism.
\end{enumerate}
\end{enumerate}
\end{teor}
\begin{proof}
$1)\Longrightarrow 2)$ By theorem \ref{prop.module heart of torsion pair} and remark \ref{rem.completion to main theorem}, we may assume  that $G$ satisfies the  standard conditions 1-5.  
So we only need to check properties 2.b and 2.c.  We proceed in several steps. All throughout the proof we put $\overline{R}=\frac{R}{t(R)}$.

\emph{Step 1:
$(\mathcal{T}\cap\overline{R}-\text{Mod},\mathcal{F})$ is the
right constituent torsion pair of a TTF triple in
$\overline{R}-\text{Mod}$}. By the standard condition 5, 
there is a morphism
$\xymatrix{h:(Q/X)^{(I)}\ar[r] &\overline{R}}$ such that if
$h'=h_{|H^{-1}(G)^{(I)}}$, then
$\text{Coker}(h')\in\overline{Gen}(V)$, where $V=H^0(G)$. But in
this case $\overline{Gen}(V)=\text{Gen}(V)=\mathcal{T}$ since
$\mathbf{t}$ is hereditary.

Put $\bar{\mathbf{b}}=\frac{\mathbf{b}}{t(R)}=\text{Im}(h')$, so
that $R/\mathbf{b}\in\mathcal{T}$. We claim that a $\bar{R}$-module
$T$ is in $\mathcal{T}$ if, and only if,
$\text{Hom}_R(\bar{\mathbf{b}},T)=0$. This will imply that
$\mathcal{T}\cap\bar{R}-\text{Mod}$ is also a torsionfree class in
$\bar{R}-\text{Mod}$. For the 'only if' part of our claim, let
$\xymatrix{f:\bar{\mathbf{b}}\ar[r]&T}$ be any morphism, where
$T\in\mathcal{T}$. We then get a pushout commutative diagram

$$\xymatrix{0 \ar[r] & \bar{\mathbf{b}} \hspace{0.1 cm}\ar@{^(->}[r]^{j} \ar[d]^{f} & \bar{R} \ar[r] \ar[d]^{g^{'}} & R/b \ar[r] \ar@{=}[d] & 0 \\ 0 \ar[r] & T \ar[r]^{\lambda} & T^{'} \ar[r] & R/b \ar[r] & 0 }$$

Then $T'\in\mathcal{T}$ and so $g^{'} \circ h_{| H^{-1}(G)^{(I)}}=0$
since $H^{-1}(G)\subseteq\text{Rej}_\mathcal{T}(Q/X)$. But $g^{'}
\circ h_{| H^{-1}(G)^{(I)}}$ is equal to the composition
$\xymatrix{H^{-1}(G)^{(I)}\ar[r]^{\hspace{0.65 cm}h'}&
\bar{\mathbf{b}}\ar[r]^{f} &T \ar[r]^{\lambda}&T'}$, which is then
the zero map. This implies that $f=0$ since $h'$ is an epimorphism
and $\lambda$ is a monomorphism. For the 'if' part, suppose that
$\text{Hom}_R(\bar{\mathbf{b}},T)=0$ and fix an epimorphism
$\xymatrix{q:\bar{R}^{(J)}\ar@{>>}[r] & T}$. Then
$q(\bar{\mathbf{b}}^{(J)})=0$, which gives an induced epimorphism
\hfill
$\xymatrix{\bar{q}:\frac{\bar{R}^{(J)}}{\bar{\mathbf{b}}^{(J)}}\cong(\frac{R}{\mathbf{b}})^{(J)}\ar@{>>}[r]
& T}$. It follows that $T\in\mathcal{T}$, which settles our claim.

\emph{Step 2: The idempotent ideal of $\overline{R}$ which defines the TTF
triple in $\overline{R}-\text{Mod}$ is
$\bar{\mathbf{b}'}=\mathbf{b}'/t(R)$, where
$\mathbf{b}'=\text{ann}_R(V/t(R)V)$}. Let
$\bar{\mathbf{b}}'=\mathbf{b}'/t(R)$ be the idempotent ideal of
$\overline{R}$ which defines the TTF triple mentioned above. We then know
(see \cite[VI.8]{S})  that $\text{Gen}(\bar{\mathbf{b}'})=\{$
$C\in\bar{R}-\text{Mod}\text{: }\bar{\mathbf{b}}'C=C\}=\{C\in\bar{R}-\text{Mod}:$
$\text{Hom}_R(C,T)=0,\text{ for all
}T\in\mathcal{T}\cap\bar{R}-\text{Mod}\}$ and
$\mathcal{T}\cap\bar{R}-\text{Mod}=\{T\in\bar{R}-\text{Mod}:$
$\bar{\mathbf{b}}'T=0\}=\text{Gen}(R/\mathbf{b}')$. In particular,
for the ideal $\mathbf{b}$ of $R$ given in the first step,   we have
that $\bar{\mathbf{b}}'\bar{\mathbf{b}}=\bar{\mathbf{b}}$ and
$\bar{\mathbf{b}}'\frac{R}{\mathbf{b}}=0$.  It follows that
$\bar{\mathbf{b}}'=\bar{\mathbf{b}}$. We then get that
$\text{Gen}(V/t(R)V)=\mathcal{T}\cap\overline{R}-\text{Mod}=\text{Gen}(R/\mathbf{b})$,
from which we deduce that
$\mathbf{b}'=\mathbf{b}=\text{ann}_R(V/t(R)V)$.

\emph{Step 3: Verification of properties 2.b and 2.c} Except for the
finite generation of $\bar{\mathbf{b}}$, property 2.b follows
immediately from the previous steps. But $R/\mathbf{b}$ is finitely
generated and we have an epimorphism
$\xymatrix{\bar{V}^n\ar@{>>}[r]&R/\mathbf{b}}$. This epimorphism
splits since  both its domain and codomain are annihilated by
$\mathbf{b}$ and $R/\mathbf{b}$ is projective in
$R/\mathbf{b}-\text{Mod}$. But $\bar{V}=V/t(R)V$ is clearly a
finitely presented $\overline{R}$-module. It follows that $R/\mathbf{b}$
is finitely presented as a left $\overline{R}$-module, which is equivalent
to saying that $\bar{\mathbf{b}}$ is finitely generated as a left ideal
of $\bar{R}=R/t(R)$. Let us fix an epimorphism $\xymatrix{\pi
:\bar{R}^{(n)}\ar@{>>}[r] &\bar{\mathbf{b}}}$. Using the canonical
map $\xymatrix{h:(Q/X)^{(I)}\ar[r] &\bar{R}}$ (see step 1), we
obtain a morphism
$\xymatrix{g:[(Q/X)^{(I)}]^{(n)}\ar[r]^{\hspace{1.0cm}h^{(n)}}
&\bar{R}^{(n)}\ar@{>>}[r]^{\pi}&\bar{\mathbf{b}}}$. If
$Y:=H^{-1}(G)$ then we have

\begin{center}
 $g[(Y^{(I)})^{(n)}]=\pi (\text{Im}(h')^{(n)})=\pi (\bar{\mathbf{b}}^{(n)})=\pi (\bar{\mathbf{b}}\bar{R}^{(n)})=
 \bar{\mathbf{b}}\bar{\mathbf{b}}=\bar{\mathbf{b}}$,
\end{center}
which proves 2.c.

$2)\Longrightarrow 1)$ It remains to prove that $G$ satisfies the standard property 5.  If
$\xymatrix{h:(\frac{Q}{X})^{(J)}\ar[r] &\bar{\mathbf{b}}}$ is the
homomorphism given in 2.c, then $h$ is an epimorphism and the
composition \linebreak  $\xymatrix{g:(\frac{Q}{X})^{(J)} \ar[r]
&\bar{\mathbf{b}} \hspace{0.02cm}\ar@{^(->}[r] &\bar{R}=R/t(R)}$ has
$R/\mathbf{b}$ as its cokernel. By property 2.b, this cokernel is in
$\mathcal{T}=\text{Gen}(V)$.

We assume in the rest of the proof that $\mathbf{t}$ is bounded.

$1), 2)\Longrightarrow 3)$ We know that
$\mathcal{T}=\text{Gen}(V)\subseteq\text{Ker}(\text{Ext}_R^1(V,?))$
and that ${}_RV$ is finitely presented. The bounded condition of
$\mathbf{t}$ implies that $R/\text{ann}_R(V)\in\mathcal{T}$, so that
$\mathbf{a}:=\text{ann}_R(V)$ annihilates all modules in
$\mathcal{T}$.
 By \cite[Proposition VI.6.12]{S}, we
know that $\mathbf{a}$ is idempotent, so that $\mathbf{t}$ is the
right constituent torsion pair of the TTF triple defined by
$\mathbf{a}$. This allows to identify $\mathcal{T}$ with
$R/\mathbf{a}-\text{Mod}$ and, using that also
$\mathcal{T}=\text{Gen}(V)\subseteq\text{Ker}(\text{Ext}_R^1(V,?))$,
we conclude that $\text{add}(V)=\text{add}(R/\mathbf{a})$. We then
get condition 3.a. We also get that $R/\mathbf{a}$ is a finitely
presented $R$-module, and so $\mathbf{a}$ is finitely generated on
the left.

The fact that $G$ satisfies the standard conditions and that $\text{Rej}_\mathcal{T}(M)=\mathbf{a}M$, for each
$R$-module $M$, automatically imply that conditions 3.b and 3.c hold. 
 Finally,
following the proof of the implication $1)\Longrightarrow 2)$, we
see that the ideal $\mathbf{b}$ obtained in assertion 2 is
identified by the properties that $\bar{\mathbf{b}}=\mathbf{b}/t(R)$
is idempotent and a $\bar{R}$-module is in $\mathcal{T}$ if, and
only if, $\mathbf{b}T=0$. Then we have $\mathbf{b}=\mathbf{a}+t(R)$
and so condition 3.d follows by using the isomorphism
$\frac{\mathbf{b}}{t(R)}=\frac{\mathbf{a}+t(R)}{t(R)}\cong\frac{\mathbf{a}}{\mathbf{a}\cap
t(R)}=\frac{\mathbf{a}}{t(\mathbf{a})}$.

$3)\Longrightarrow 1)$ Since we have
$\text{Rej}_\mathcal{T}(M)=\mathbf{a}M$, for each $R$-module $M$, it
is easily verified that $G$ satisfies all the standard conditions 1-5.
\end{proof}

\begin{cor} \label{cor.faithful hereditary pair}
If $\mathbf{t}=(\mathcal{T},\mathcal{F})$ is a faithful hereditary
torsion pair such that its heart $\mathcal{H}_\mathbf{t}$ is a
module category, then $\mathbf{t}$ is the right constituent pair of
a TTF triple in $R-\text{Mod}$ defined by an idempotent ideal
 $\mathbf{a}$ which is finitely generated on the left.
\end{cor}

\begin{cor} \label{cor.hereditary pairs in commutative}
Let $R$ be a commutative  ring and let
$\mathbf{t}=(\mathcal{T},\mathcal{F})$ be a hereditary torsion pair
in $R-\text{Mod}$. The heart $\mathcal{H}_\mathbf{t}$ is a module
category if, and only if,  $\mathbf{t}$ is  (left or right)
constituent pair of a centrally split TTF triple in $R-\text{Mod}$.
In that case $\mathcal{H}_\mathbf{t}$ is equivalent to
$R-\text{Mod}$.
\end{cor}
\begin{proof}
Since $\mathbf{t}$ is bounded, last theorem says that $\mathbf{t}$
is the right constituent torsion pair of a TTF triple in
$R-\text{Mod}$ defined by an idempotent ideal $\mathbf{a}$ which is
finitely generated.
 But each finitely generated
idempotent ideal of a commutative ring is generated by an idempotent
element (see the proof of Lemma VI.8.6 in \cite{S}). Then the TTF
triple is centrally split. By Corollary \ref{cor.split and
commutative case} below, we have that $\mathcal{H}_\mathbf{t}$ is
equivalent to $R-\text{Mod}$.
\end{proof}

\section{When the progenerator is a sum of stalk complexes}

Recall that if $M$ and $N$ are $R$-modules, then
$\text{Ext}_R^1(M,N)=\text{Hom}_{\mathcal{D}(R)}(M,N[1])$ has a
canonical structure of $\text{End}_R(N)-\text{End}_R(M)-$bimodule
given by composition of morphisms in $\mathcal{D}(R)$. But then it
has also a structure of $\text{End}_R(M)^{op}-\text{End}_R(N)^{op}$,
by defining $\alpha^o\cdot\epsilon\cdot
f^o=f\circ\epsilon\circ\alpha$, for all $\alpha\in\text{End}_R(M)$
and $f\in\text{End}_R(N)$.

It is natural to expect that the 'simplest' case in which the heart
is a module category appears  when the progenerator of the heart can
be chosen to be a sum of stalk complexes. Our next result gives
criteria for that to happen.

\begin{prop} \label{prop.progenerator as sum of stalks}
The following assertions are equivalent:

\begin{enumerate}
\item $\mathcal{H}_\mathbf{t}$ has a progenerator of the form $V[0]\oplus
Y[1]$, where $V\in\mathcal{T}$ and $Y\in\mathcal{F}$;
\item There are $R$-modules  $V$ and $Y$ satisfying the following
properties:

\begin{enumerate}
\item $V$ is finitely presented and
$\mathcal{T}=\text{Pres}(V)\subseteq\text{Ker}(\text{Ext}_R^1(V,?))$;
\item $\text{Ext}_R^2(V,?)$ vanishes on $\mathcal{F}$;
\item $Y$ is a finitely generated projective $R/t(R)$-module which is in
${}^\perp\mathcal{T}$;
\item For each $F\in\mathcal{F}$, the module
$\frac{F/tr_Y(F)}{t(F/tr_Y(F))}$ embeds into a module in
$\mathcal{T}$, where $tr_Y(F)$ denote the trace of $Y$ in $F$.
\end{enumerate}
\end{enumerate}

In this case $\mathcal{H}_\mathbf{t}$ is equivalent to
$S-\text{Mod}$, where $S=\begin{pmatrix}\text{End}_R(Y)^{op} & 0\\
\text{Ext}_R^1(V,Y) & \text{End}_R(V)^{op}
\end{pmatrix}$, when viewing $\text{Ext}_R^1(V,Y)$ as a
$\text{End}_R(V)^{op}-\text{End}_R(Y)^{op}-$bimodule in the usual
way.
\end{prop}
\begin{proof}
By \cite[Theorem 4.8]{PS} and by condition 2.a, all throughout the
proof we can assume that $\mathcal{F}$ is closed under taking direct
limits in $R-\text{Mod}$.

$1)\Longrightarrow 2)$ Put $G=V[0]\oplus Y[1]$. By  lemma
\ref{lem.description  de t}, we get condition 2.a. On the other
hand, the projective condition of $V[0]$ in $\mathcal{H}_\mathbf{t}$
implies that
$0=\text{Ext}_{\mathcal{H}_\mathbf{t}}^1(V[0],F[1])=\text{Ext}_R^2(V,F)$,
for all $F\in\mathcal{F}$. Then condition 2.b also holds.

The projective condition of $Y[1]$ in $\mathcal{H}_\mathbf{t}$
implies that
$0=\text{Ext}_{\mathcal{H}_\mathbf{t}}^1(Y[1],F[1])=\text{Ext}_R^1(Y,F)$,
for all $F\in\mathcal{F}$ and that
$0=\text{Ext}_{\mathcal{H}_\mathbf{t}}^1(Y[1],T[0])=\text{Hom}_{\mathcal{D}(R)}(Y[1],T[1])
\cong\text{Hom}_R(Y,T)=0$, for all $T\in\mathcal{T}$. Then we have
that $Y\in {}^\perp\mathcal{T}$. Moreover, if
$\xymatrix{f:\frac{R}{t(R)}^{(I)}\ar@{>>}[r] & Y}$ is any
epimorphism, then $f$ is a retraction, which implies that $Y$ is a
projective $R/t(R)$-module. The fact that $Y$ is finitely generated
follows from the compactness of $Y[1]$ in $\mathcal{H}_\mathbf{t}$
since then  $\text{Hom}_R(Y,?)_{|
\mathcal{F}}\cong\text{Hom}_{\mathcal{D}(R)}(Y[1],?[1])_{|
\mathcal{F}}$ preserves coproducts of modules in $\mathcal{F}$. We
then get condition 2.c.

For each $F\in\mathcal{F}$, let us consider the canonical morphism
$\xymatrix{g:Y^{(\text{Hom}_R(Y,F))}\ar[r] & F}$. We then get the
morphism
$\xymatrix{Y[1]^{(\text{Hom}_{\mathcal{H}_\mathbf{t}}(Y[1],F[1]))}\cong
Y^{(\text{Hom}_R(Y,F))}[1]\ar[r]^{\hspace{2.6 cm}g[1]}& F[1]}$ whose
image is the trace of $Y[1]$ in $F[1]$ within the category
$\mathcal{H}_\mathbf{t}$. The cokernel of $g[1]$ is precisely the
stalk complex
$\frac{\text{Coker}(g)}{t(\text{Coker}(g))}[1]=\frac{F/tr_Y(F)}{t(F/tr_Y(F))}[1]$.
Due to the projectivity of $Y[1]$ in $\mathcal{H}_\mathbf{t}$, we
have that
$\text{Hom}_R(Y,\frac{\text{Coker}(g)}{t(\text{Coker}(g))})\cong
\text{Hom}_{\mathcal{H}_\mathbf{t}}(Y[1],\frac{\text{Coker}(g)}{t(\text{Coker}(g))}[1])=0$.
The fact that $V[0]\oplus Y[1]$ is a projective generator of
$\mathcal{H}_\mathbf{t}$ implies then that the canonical morphism
\hfill
$\xymatrix{q:V[0]^{(\text{Hom}_{\mathcal{H}_\mathbf{t}}(V[0],\frac{\text{Coker}(g)}{t(\text{Coker}(g))}[1]))}
\ar[r]&\frac{\text{Coker}(g)}{t(\text{Coker}(g))}[1]}$ is an
epimorphism in $\mathcal{H}_{\mathbf{t}}$. We necessarily have
$\text{Ker}(q)=T[0]$, for some $T\in\mathcal{T}$. Condition 2.d
follows then from the long exact sequence of homologies associated
to the triangle

\begin{center}
$\xymatrix{T[0] \ar[r] &
V[0]^{(\text{Hom}_{\mathcal{H}_\mathbf{t}}(V[0],\frac{\text{Coker}(g)}{t(\text{Coker}(g))}[1]))}
\ar[r]^{\hspace{1.25cm}q}
&\frac{\text{Coker}(g)}{t(\text{Coker}(g))}[1] \ar[r]^{\hspace{0.8
cm}+} &}$
\end{center}

$2)\Longrightarrow 1)$ From conditions 2.a and 2.b we deduce that
$\text{Ext}_{\mathcal{H}_\mathbf{t}}^1(V[0],?)$ vanishes on  stalk
complexes $T[0]$ and $F[1]$, for each $T\in\mathcal{T}$ and
$F\in\mathcal{F}$. Similarly,  from condition 2.c we deduce that
$\text{Ext}_{\mathcal{H}_\mathbf{t}}^1(Y[1],?)$ vanishes on all
those stalk complexes. It follows that $G:=V[0]\oplus Y[1]$ is a
projective object of $\mathcal{H}_\mathbf{t}$.

Knowing that $G$ is a projective object, in order to prove that $G$
is a generator of of $\mathcal{H}_\mathbf{t}$, we just need to prove
that it generates all stalk complexes $X$, with
$X\in\mathcal{T}[0]\cup\mathcal{F}[1]$. Note that from condition 2.a
we get that $V[0]$ generates all stalk complexes $T[0]$ and, hence,
that $\mathcal{T}[0]\subseteq\text{Gen}_{\mathcal{H}_\mathbf{t}}(G)$. If now we take
$F\in\mathcal{F}$, then the argument in the proof of the other
implication shows that the canonical morphism
$\xymatrix{Y[1]^{(\text{Hom}_{\mathcal{H}_\mathbf{t}}(Y[1],F[1]))}\cong
Y^{(\text{Hom}_R(Y,F))}[1]\ar[r]^{\hspace{2.65 cm}g[1]} &F[1]}$ has
as cokernel $F'[1]$, where $F'=\frac{F/tr_Y(F)}{t(F/tr_Y(F))}$. By
hypothesis we have a monomorphism $F'\rightarrowtail T$ and, hence,
an exact sequence $\xymatrix{0\ar[r] & F' \ar[r] &T \ar[r] &T'\ar[r]
&0}$, where $T$ and $T'$ are in $\mathcal{T}$. We then get an exact
sequence in $\mathcal{H}_\mathbf{t}$:

\begin{center}
$\xymatrix{0\ar[r] & T[0] \ar[r] &T'[0] \ar[r] &F'[1]\ar[r] &0}$
\end{center}
which shows that $F'[1]$ is generated by $V[0]$ and, hence, that
$F'[1]\in\text{Gen}_{\mathcal{H}_\mathbf{t}}(G)$. But we have an
exact sequence in $\mathcal{H}_\mathbf{t}$

\begin{center}
$\xymatrix{0\ar[r]
&\text{Im}_{\mathcal{H}_\mathbf{t}}(g[1])\ar[r]&F[1]\ar[r] &
F'[1]\ar[r]& 0}$
\end{center}
Then we have that
$\text{Im}_{\mathcal{H}_\mathbf{t}}(g[1])\in\text{Gen}_{\mathcal{H}_\mathbf{t}}(Y[1])\subseteq\text{Gen}_{\mathcal{H}_\mathbf{t}}(G)$
and
$F'[1]\in\text{Gen}_{\mathcal{H}_\mathbf{t}}(V[0])\subseteq\text{Gen}_{\mathcal{H}_\mathbf{t}}(G)$.
The projective condition of $G$ in $\mathcal{H}_\mathbf{t}$ proves
now that also $F[1]\in\text{Gen}_{\mathcal{H}_\mathbf{t}}(G)$.
Hence $G$ is a generator of $\mathcal{H}_\mathbf{t}$.

We finally prove that $G$ is compact in $\mathcal{H}_\mathbf{t}$,
which is equivalent to proving that $V[0]$ and $Y[1]$ are compact in
this category. For each family $(M_i)_{i\in I}$ of objects in
$\mathcal{H}_\mathbf{t}$, we have a family of exact sequences
in $\mathcal{H}_\mathbf{t}$:

\begin{center}
$\xymatrix{0\ar[r] & H^{-1}(M_i)[1]\ar[r] &M_i \ar[r] & H^0(M_i)[0]
\ar[r] & 0 & (i\in I)}$
\end{center}
Using this and the projectivity of $V[0]$ and $Y[1]$, the task is
reduced to check the following facts:

\begin{enumerate}
\item[i)] $\text{Hom}_R(Y,?)$ preserves coproducts of modules in $\mathcal{F}$;
\item[ii)] $\text{Hom}_R(V,?)$ preserves coproducts of
modules in $\mathcal{T}$;
\item[iii)] $\text{Ext}_R^1(V,?)$ preserves coproducts of
modules in $\mathcal{F}$.
\end{enumerate}
Conditions i) and ii) automatically hold since $Y$ and $V$ are
finitely generated modules. Condition iii) follows from the fact
that $V$ is finitely presented.

The final statement of the proposition is clear, because the ring $S=\begin{pmatrix}\text{End}_R(Y)^{op} & 0\\
\text{Ext}_R^1(V,Y) & \text{End}_R(V)^{op}
\end{pmatrix}$ is isomorphic to $\text{End}_{\mathcal{H}_\mathbf{t}}(V[0]\oplus
Y[1])^{op}$.
\end{proof}

We have now the following consequences of last proposition.

\begin{cor} \label{cor.progenerator stalk in 0}
Let $V$ be an $R$-module and consider the following conditions

\begin{enumerate}
\item $V$ is a classical 1-tilting module;
\item $\mathbf{t}=(\text{Pres}(V),\text{Ker}(\text{Hom}_R(V,?)))$ is a torsion pair in $R-\text{Mod}$ and $V[0]$ is a progenerator
of $\mathcal{H}_\mathbf{t}$;
\item $V$  finitely presented  and satifies the following
conditions:

\begin{enumerate}
\item
$\mathcal{T}:=\text{Pres}(V)=\text{Gen}(V)\subseteq\text{Ker}(\text{Ext}_R^1(V,?))$;
\item $\text{Ext}_R^2(V,?)$ vanishes on $\mathcal{F}:=\text{Ker}(\text{Hom}_R(V,?))$;
\item Each module of $\mathcal{F}$ embeds into a module of
$\mathcal{T}$.
\end{enumerate}
\end{enumerate}
Then the implications $1)\Longrightarrow 2)\Longleftrightarrow 3)$
hold true. Moreover, when conditions 2 or 3 hold, $\mathbf{t}$ is
also a
 torsion pair in the Grothendieck category
$\mathcal{G}:=\overline{Gen}(V)$, $V$ is a classical 1-tilting
object of $\mathcal{G}$ and the canonical functor
$\xymatrix{\mathcal{D}(\mathcal{G})\ar[r] &\mathcal{D}(R)}$ gives by
restriction an equivalence of categories
$\xymatrix{\mathcal{H}_\mathbf{t}(\mathcal{G})
\ar[r]^{\hspace{0.2cm}\sim} &\mathcal{H}_\mathbf{t}}$, where
$\mathcal{H}_\mathbf{t}(\mathcal{G})$ is the heart of the torsion
pair in $\mathcal{G}$.
\end{cor}
\begin{proof}
$1)\Longrightarrow 2)$ is a particular case of  \cite[Proposition
5.3]{PS}.

$2)\Longrightarrow 3)$ is a direct consequence of  proposition
\ref{prop.progenerator as sum of stalks}.

$3)\Longrightarrow 2)$ We  need to prove that
$\mathcal{T}=\text{Gen}(V)$ is closed under taking extensions in
$R-\text{Mod}$. In that case
$\mathbf{t}=(\text{Gen}(V),\text{Ker}(\text{Hom}_R(V,?)))$ is a
torsion pair in $R$-Mod and the implication will follow from
proposition \ref{prop.progenerator as sum of stalks}. Let
$0\rightarrow T\longrightarrow M\longrightarrow T'\rightarrow 0$ is
an exact sequence in $R-\text{Mod}$, with $T,T'\in\mathcal{T}$. We
want to prove that $M\in\mathcal{T}$. By pulling  back the exact
sequence along an epimorphism $\xymatrix{p:V^{(I)}\ar@{>>}[r] &
T'}$, we can assume without loss of generality that $T'=V^{(I)}$.
But in this case the sequence splits since
$\text{Ext}_R^1(V^{(I)},T')\cong\text{Ext}_R^1(V,T)^I=0$.

Let us prove now the final statement. By lemma \ref{lem.description
de t}, we know that $V$ is classical quasi-tilting. It essentially
follows from the arguments in \cite[Section 2]{CDT1}  that  $V$ is a
classical 1-tilting object of $\mathcal{G}:=\overline{Gen}(V)$. But
it also follows from something stronger that we need,  namely, that
the canonical map $\xymatrix{\varphi :\text{Ext}_\mathcal{G}
^1(V,X)\ar[r] &\text{Ext}_R ^1(V,X)}$ is an isomorphism, for each
$X\in\mathcal{G}$. It is clearly injective. To prove the
surjectivity, let $\xymatrix{0 \ar[r] & X \ar[r] & M\ar[r] &V
\ar[r]& 0}$ (*) be an exact sequence in $R\text{-Mod}$. Recall that
the injective objects of $\mathcal{G}$ are modules in
$\text{Gen}(V)=\mathcal{T}$ (see \cite[Introduction]{GG}). This
implies that we have a monomorphism $\xymatrix{u:X
\hspace{0.02cm}\ar@{^(->}[r] & T}$, with $T\in\mathcal{T}$. By
pushing out the sequence (*) along the monomorphism $u$ and using
the fact that $\text{Ext}_R^1(V,T)=0$, we get a monomorphism
$\xymatrix{M\hspace{0.04cm}\ar@{^(->}[r] &T\oplus V}$, which implies
that $M\in\mathcal{G}$. Then the sequence (*) lives in $\mathcal{G}$
and, hence, $\varphi$ is an isomorphism.

 On the other
hand, the inclusion functor
$\xymatrix{\mathcal{G}\hspace{0.02cm}\ar@{^(->}[r]& R\text{-Mod}}$
is exact and, hence, extends to a triangulated functor
$\xymatrix{j:\mathcal{D}(\mathcal{G})\ar[r] &\mathcal{D}(R)}$,
which need be neither faithful nor full, but induces by restriction
a functor
$\xymatrix{\tilde{j}:\mathcal{H}_\mathbf{t}(\mathcal{G})\ar[r]
&\mathcal{H}_\mathbf{t}:=\mathcal{H}_\mathbf{t}(R-\text{Mod})}$. We
claim that, up to natural isomorphism, the following diagram of
functors is commutative, where $S=\text{End}_R(V)^{op}$:

$$\xymatrix{\mathcal{H}_{\mathbf{t}}(\mathcal{G}) \ar[rd]_{Hom_{D(\mathcal{G})}(V[0],?) \hspace{0.2cm}} \ar[rr]^{\tilde{j}} & & \mathcal{H}_{\mathbf{t}} \ar[dl]^{\hspace{0.1 cm}Hom_{D(R)}(V[0],?)} \\& S-Mod}$$

Due to the projective condition of $V[0]$ both in
$\mathcal{H}_\mathbf{t}(\mathcal{G})$ and $\mathcal{H}_\mathbf{t}$,
we just need to see that the maps induced by the functor $j$:

\begin{center}
$\xymatrix{\text{Hom}_\mathcal{G}(V,T)\cong
\text{Hom}_{\mathcal{D}(G)}(V[0],T[0])\ar[r]&
\text{Hom}_{\mathcal{D}(R)}(V[0],T[0])\cong\text{Hom}_R(V,T)}$

$\xymatrix{\text{Ext}_\mathcal{G}^1(V,F)\cong\text{Hom}_{\mathcal{D}(G)}(V[0],F[1])
\ar[r]&\text{Hom}_{\mathcal{D}(R)}(V[0],F[1])\cong\text{Ext}_R^1(V,F)}$
\end{center}
are isomorphisms. The first one is clear and the second one has been
proved in the previous paragraph. By assertion 2, the functor
$\xymatrix{\text{Hom}_{\mathcal{H}_\mathbf{t}}(V[0],?):\mathcal{H}_\mathbf{t}\ar[r]
& S\text{-Mod}}$ is an equivalence of categories. Since $V$ is a
classical 1-tilting object of $\mathcal{G}$, the functor
$\xymatrix{\text{Hom}_{\mathcal{H}_\mathbf{t}(\mathcal{G})}(V[0],?):\mathcal{H}_\mathbf{t}(\mathcal{G})\ar[r]&S\text{-Mod}}$
is also an equivalence (see \cite[Proposition 5.3]{PS}). It follows
that
$\xymatrix{\tilde{j}:\mathcal{H}_\mathbf{t}(\mathcal{G})\ar[r]&\mathcal{H}_\mathbf{t}}$
is an equivalence of categories.
\end{proof}

The following is now very natural.

\begin{ques} \label{ques.classical tilting pair}
Let $\mathbf{t}=(\mathcal{T},\mathcal{F})$ be a torsion pair in
$R\text{-Mod}$ satisfying the equivalent conditions 2 and 3 of
corollary \ref{cor.progenerator stalk in 0}. Is $\mathbf{t}$ a
classical tilting torsion pair?
\end{ques}

\begin{lemma} \label{lem.quasi-tilting implies projective}
Let $V$ be a classical quasi-tilting $R$-module such that
$\text{Gen}(V)$ is closed under submodules and let $t(R)$ be the
trace of $V$ in $R$. An endomorphism $\beta$ of $V$ satisfies that
$\text{Im}(\beta )\subseteq t(R)V$ if, and only if, it factors
through a (finitely generated) projective $R$-module.
\end{lemma}
\begin{proof}
We put $\mathbf{t}=(\text{Gen}(V),\text{Ker}(\text{Hom}_R(V,?)))$,
which is a hereditary torsion pair.  The 'if' part is clear.
Conversely, suppose that $\text{Im}(\beta )\subseteq t(R)V$. Let
$\xymatrix{q:V^{(\text{Hom}_R(V,R))}\ar@{>>}[r]& t(R)=tr_V(R)}$,
$\xymatrix{i:t(R)\hspace{0.03cm}\ar@{^(->}[r] &R}$, $\xymatrix{\pi
:R^{(V)} \ar@{>>}[r]& V}$ and $\xymatrix{j:t(R)V \hspace{0.04cm}
\ar@{^(->}[r]&  V}$ be the canonical morphisms and  let
$\xymatrix{\pi':t(R)^{(V)}\ar@{>>}[r] &t(R)V}$ be the epimorphism
given by the restriction of $\pi$ to $t(R)^{(V)}$. We have a
commutative diagram

$$\xymatrix{V^{(Hom_{R}(V,R) \times V)} \ar@{>>}[dr]^{\rho}\ar@{>>}[r]^{\hspace{0.7cm}q^{(V)}} &t(R)^{(V)} \ar[r]^{i^{(V)}} \ar@{>>}[d]^{\pi^{'}}& R^{(V)} \ar[d]^{\pi} \\& t(R)V \ar@{^(->}[r]^{j} & V}$$

where $\rho :=\pi'\circ q^{(V)}$. We have a factorization
$j\circ\tilde{\beta}=\beta$, where
$\tilde{\beta}\in\text{Hom}_R(V,t(R)V)$. Due to the hereditary
condition of $\mathbf{t}$, we know that $\text{Ker}(\rho
)\in\mathcal{T}\subseteq\text{Ker}(\text{Ext}_R^1(V,?))$, which
implies $\tilde{\beta}$ factors through $\rho$. Fix a morphism
$\xymatrix{\gamma :V\ar[r] & V^{(\text{Hom}_R(V,R)\times V)}}$ such
that $\tilde{\beta}=\rho\circ\gamma$. Then we have:

\begin{center}
$\beta =j\circ\tilde{\beta}=j\circ\rho\circ\gamma =\pi\circ
i^{(V)}\circ q^{(V)}\circ\gamma$,
\end{center}
so that $\beta$ factors through $R^{(V)}$.
\end{proof}

\begin{cor} \label{cor.progenerator sum-of-stalks hereditary}
Let us assume that $\mathbf{t}=(\mathcal{T},\mathcal{F})$ is a
hereditary torsion pair in $R\text{-Mod}$. The following assertions
are equivalent:

\begin{enumerate}
\item $\mathcal{H}_\mathbf{t}$ has a progenerator of the form $V[0]\oplus
Y[1]$, where $V\in\mathcal{T}$ and $Y\in\mathcal{F}$;
\item There are $R$-modules  $V$ and $Y$ satisfying the following
properties:

\begin{enumerate}
\item $V$ is finitely presented and
$\mathcal{T}=\text{Pres}(V)\subseteq\text{Ker}(\text{Ext}_R^1(V,?))$;
\item $\text{Ext}_R^2(V,?)$ vanishes on $\mathcal{F}$;
\item $Y$ is a finitely generated projective $R/t(R)$-module which is in
${}^\perp\mathcal{T}$;
\item For each $F\in\mathcal{F}$, the module $F/tr_Y(F)$ is in
$\mathcal{T}$,  where $tr_Y(F)$ denotes the trace of $Y$ in $F$.
\end{enumerate}
\end{enumerate}
In this case, $\mathbf{t}':=(\mathcal{T}\cap\frac{R}{t(R)}-\text{Mod},\mathcal{F})$
is a torsion pair in $\frac{R}{t(R)}-Mod$ which is the right
constituent of a TTF triple in this category and has the property
that $\frac{V}{t(R)V}[0]\oplus Y[1]$ is a progenerator of
$\mathcal{H}_{\mathbf{t}'}$. Moreover, the forgetful functor
$\xymatrix{\mathcal{H}_{\mathbf{t}'}\ar[r] &\mathcal{H}_\mathbf{t}}$
is faithful.
\end{cor}
\begin{proof}
All throughout the proof we put $I=t(R)$ and $\bar{M}=M/IM$, for each $R$-module $M$.
The equivalence of assertions 1 and 2 is a direct consequence of
proposition \ref{prop.progenerator as sum of stalks}. From
theorem \ref{teor.hereditary case} and its proof, we know that
$(\mathcal{T}\cap\bar{R}-\text{Mod},\mathcal{F})$ is the right
constituent torsion pair of a TTF triple
$(\mathcal{C}_I,\mathcal{T}_I,\mathcal{F})$  in
$\bar{R}-\text{Mod}$. Moreover, by property 2.c, the class
$\text{Ker}(\text{Hom}_{\bar{R}}(Y,?))$  contains
$\mathcal{T}\cap\bar{R}-\text{Mod}$ and  is closed
 under taking quotients. Using property 2.d, it then follows that
 the inclusion
 $\text{Ker}(\text{Hom}_{\bar{R}}(Y,?))\subseteq\mathcal{T}\cap\bar{R}-\text{Mod}$ also holds,
 which implies that $\mathcal{C}_I=\text{Gen}(Y)$. If now $\mathbf{a}$ is the two-sided  ideal of $R$ given by the
equality $\bar{\mathbf{a}}=\frac{\mathbf{a}}{I}=tr_Y(\frac{R}{I})$,
then $\bar{\mathbf{a}}$ is the idempotent ideal of $\bar{R}$ which
defines the TTF triple and, by the proof of theorem
\ref{teor.hereditary case}, we know that
$\mathbf{a}=\text{ann}_R(\bar{V})$ and that
$\text{add}(\bar{V})=\text{add}(R/\mathbf{a})$, so that $\bar{V}$ is
a progenerator of $\frac{R}{\mathbf{a}}-\text{Mod}$.

Now the $\bar{R}$-modules $\bar{V}$ and $Y$ satisfy the conditions
2.a, 2.c and 2.d with respect to the torsion pair
$\mathbf{t}'=(\mathcal{T}_I,\mathcal{F})$ of $\bar{R}-\text{Mod}$.
On the other hand,  $\mathbf{t}$ and $\mathbf{t}'$ are hereditary
torsion pairs in $R-\text{Mod}$ and $\bar{R}-\text{Mod}$,
respectively. Then, for each $F\in\mathcal{F}$,  the injective
envelope $E(F)$ in $R-\text{Mod}$ is also in $\mathcal{F}$ (see
\cite[Proposition VI.3.2]{S}). In particular, we have that
$E(F)\in\bar{R}-\text{Mod}$, so that $E(F)$ is also the injective
envelope of $F$ as a $\bar{R}$-module and, hence, the first cosyzygy
$\Omega^{-1}(F)$ is the same in $R-\text{Mod}$ and
$\bar{R}-\text{Mod}$. In order to check condition 2.b for $\bar{V}$,
we need to check that
$\text{Ext}_{\bar{R}}^{1}(\bar{V},\Omega^{-1}(F))=0$. But, using
condition 2.b for $V$,  our needed goal will follow from
 something stronger that we will prove.
Namely, that if $\xymatrix{p=p_V:V\ar@{>>}[r]&\bar{V}}$ is the
canonical projection, then the composition

\begin{center}
$\xymatrix{\varphi
:\text{Ext}^1_{\bar{R}}(\bar{V},M)\ar[r]^{\hspace{0.2cm}can}&\text{Ext}^1_R(\bar{V},M)\ar[rr]^{\text{Ext}_R^1(p,M)}&&
\text{Ext}^1_R(V,M)}$
\end{center}
is a monomorphism, for all $M\in\bar{R}-\text{Mod}$.

Let $\xymatrix{0 \ar[r] & M \ar[r]^{j} &N \ar[r]^{q} & \bar{V}
\ar[r]& 0}$ be an exact sequence in $\bar{R}-\text{Mod}$ which
represents an element of $\text{Ker}(\varphi)$. Then the projection
$\xymatrix{p:V\ar@{>>}[r]&\bar{V}}$ factors through $q$. Fixing a
morphism $\xymatrix{g:V\ar[r] & N}$ such that $q\circ g=p$ and
taking into account that $IN=0$, we get a morphism
$\xymatrix{\bar{g}:\bar{V}\ar[r] &N}$ which is a section for $q$.

In order to prove the final assertion, with the notation of the
previous lemma, consider the following composition of morphisms of
abelian groups, where $F\in\mathcal{F}$

\begin{center}
$\xymatrix{\text{Ext}_R^1(j\circ\rho
,F):\text{Ext}_R^1(V,F)\ar[rr]^{\hspace{0.7cm}\text{Ext}_R^1(j,F)}
&&\text{Ext}_R^1(t(R)V,F)\ar[rr]^{\text{Ext}_R^1(\rho
,F)\hspace{0.65cm}} &&\text{Ext}_R^1(V^{(\text{Hom}_R(V,R)\times
V)},F)}$.
\end{center}
We have that $\text{Ext}_R^1(\rho ,F)$ is a monomorphism, because
$\text{Ker}(\rho )\in\mathcal{T}$ and hence
$\text{Hom}_R(\text{Ker}(\rho ),F)=0$. But
$\text{Ext}_R^1(j\circ\rho ,F)=0$ since $j\circ\rho$ factors through
a projective $R$-module. We then get that $\text{Ext}_R^1(j,F)$ is
the zero map, for each $F\in\mathcal{F}$. By considering the
canonical exact sequence $\xymatrix{0 \ar[r] & t(R)V
\hspace{0.03cm}\ar@{^(->}[r]^{\hspace{0.3cm}j} & V \ar@{>>}[r]&
\bar{V} \ar[r] & 0}$ and applying to it the long exact sequence of
$\text{Ext}(?,F)$, we get:

\begin{center}
$\xymatrix{0=\text{Hom}_R(t(R)V,F)\ar[r]&\text{Ext}_R^1(\bar{V},F)\ar[r]&\text{Ext}_R^1(V,F)\ar[r]^{0\hspace{0.4cm}}&\text{Ext}_R^1(t(R)V,F)}$
\end{center}
which proves that
$\text{Ext}_R^1(\bar{V},F)\cong\text{Ext}_R^1(V,F)$, for each
$F\in\mathcal{F}$. Moreover, by the two previous paragraphs, we get
that the map
$\xymatrix{\text{Ext}^1_{\bar{R}}(\bar{V},F)\ar[r]^{can}&\text{Ext}^1_R(\bar{V},F)}$
is a monomorphism.

Let us put $\bar{G}:=\bar{V}[0]\oplus Y[1]$. We claim that the map
$\xymatrix{\text{Hom}_{\mathcal{H}_{t'}}(\bar{G},M)\ar[r]
&\text{Hom}_{\mathcal{H}_{t}}(\bar{G},M)}$ is injective, for all
$M\in\mathcal{H}_{t'}$. Bearing in mind that we have  isomorphisms
of abelian groups

\begin{center}
$\text{Hom}_{\mathcal{H}_{t'}}(Y[1],M)\cong\text{Hom}_{\bar{R}}(Y,H^{-1}(M))=\text{Hom}_R(Y,H^{-1}(M))\cong\text{Hom}_{\mathcal{H}_{t}}(Y[1],M)$,
\end{center}
our task reduces to check that the canonical map
$\xymatrix{\text{Hom}_{\mathcal{H}_{t'}}(\bar{V}[0],M)\ar[r]&\text{Hom}_{\mathcal{H}_{t}}(\bar{V}[0],M)}$
is injective. But we have the following commutative diagram:

$$\xymatrix{0 \ar[r] & Ext_{\bar{R}}^{1}(\bar{V},H^{-1}(M)) \ar[r] \ar@{^(->}[d]^{can}& Hom_{\mathcal{H}_{\mathbf{t}^{'}}}(\bar{V}[0],M) \ar[r] \ar[d] & Hom_{\bar{R}}(\bar{V},H^{0}(M)) \ar[r] \ar[d]^{\wr}& Ext_{\bar{R}}^{2}(\bar{V},F)=0\\ 0 \ar[r] & Ext_{R}^{1}(\bar{V},H^{-1}(M)) \ar[r] & Hom_{\mathcal{H}_{\mathbf{t}}}(\bar{V}[0],M) \ar[r] & Hom_{R}(\bar{V},H^{0}(M)) \ar[r] & Ext_{R}^{2}(\bar{V},F) }$$

The right vertical arrow is an isomorphism since $H^0(M)$ is a
$\bar{R}$-module, and the left vertical arrow is a monomorphism. It
then follows that the central vertical arrow  is a monomorphism, as
desired.

Let us fix any object $M\in\mathcal{H}_{\mathbf{t}'}$ and consider
the full subcategory $\mathcal{C}_M$ of $\mathcal{H}_{\mathbf{t}'}$
consisting of the objects $N$ such that the canonical map
$\xymatrix{\text{Hom}_{\mathcal{H}_{t'}}(N,M)\ar[r]
&\text{Hom}_{\mathcal{H}_{t}}(N,M)}$ is a monomorphism. This
subcategory is closed under taking coproducts and cokernels and, by
the previous paragraph, it contains $\bar{G}$. We then have
$\mathcal{C}_M=\mathcal{H}_{\mathbf{t}'}$ and, since this is true
for any $M\in\mathcal{H}_{\mathbf{t}'}$, we conclude that the
forgetful functor
$\xymatrix{\mathcal{H}_{\mathbf{t}'}\ar[r]&\mathcal{H}_\mathbf{t}}$
is faithful.
\end{proof}

\begin{rem}
It can be easily derived from the proof of corollary
\ref{cor.progenerator sum-of-stalks hereditary} that the functor
$\xymatrix{\mathcal{H}_{\mathbf{t}'}\ar[r]&\mathcal{H}_{\mathbf{t}}}$
is full if, and only if, each exact sequence $\xymatrix{0 \ar[r] & Y
\ar[r] & M \ar[r] & V/IV \ar[r] & 0}$ in $R-\text{Mod}$ satisfies
that $IM=0$.
\end{rem}

\begin{prop} \label{prop.non-HKM modular heart}
Let $\mathbf{t}=(\mathcal{T},\mathcal{F})$ be hereditary and suppose
that it is the left constituent torsion pair of a TTF triple in $R-\text{Mod}$. Then
$\mathcal{H}_\mathbf{t}$ is a module category if, and only if, there
is a finitely generated projective $R$-module $P$ such that
$\mathcal{T}=\text{Gen}(P)$. In such case, the following assertions
hold:

\begin{enumerate}
\item $\mathbf{t}$ is HKM if, and only if, there is a finitely generated projective $R$-module $Q'$ such that $\text{Hom}_R(Q',P)=0$ and
$\text{add}(\frac{Q'}{t(Q')})=\text{add}(\frac{R}{t(R)})$. In
general, $\mathbf{t}$ need not be an HKM torsion pair;
\item $\mathbf{t}$ is the right constituent of a TTF triple in $R-\text{Mod}$ if, and only if, $P$ is finitely generated over its
endomorphism ring.
\end{enumerate}
\end{prop}
\begin{proof}
 \emph{'If' part}: Let $P$ be a finitely generated projective
$R$-module such that $\mathcal{T}=\text{Gen}(P)$.  We will check
that  $V=P$ and $Y=\frac{R}{t(R)}$ satisfy conditions 2.a-d of
corollary \ref{cor.progenerator sum-of-stalks hereditary}. All these
properties  are trivially satisfied, except the fact that
$Y\in{}^\perp\mathcal{T}$. For that, we consider the TTF triple
$(\mathcal{T},\mathcal{F},\mathcal{F}^\perp)$. By \cite[Lemma
VI.8.3]{S}, we know that $\mathcal{T}\subseteq\mathcal{F}^\perp$. It
particular,
$Y=R/t(R)\in\mathcal{F}={}^\perp(\mathcal{F}^\perp)\subseteq
{}^\perp\mathcal{T}$.

 \emph{'Only if' part}: Let $\mathbf{a}$ be the idempotent ideal
 which defines the TTF triple, so that $\mathcal{T}=\{T\in R-\text{Mod}:$
 $\mathbf{a}T=T\}$. By theorem \ref{prop.module heart of torsion
 pair}, we have a progenerator $G:= \cdots \longrightarrow 0\longrightarrow X\stackrel{j}{\longrightarrow}Q\stackrel{d}{\longrightarrow}P\longrightarrow
 0 \longrightarrow \cdots$, where $P$ and $Q$ are finitely generated projective and
 $\mathcal{T}=\text{Gen}(V)$, where $V=H^0(G)$. We then have
 $\mathbf{a}M=t(M)=tr_V(M)$, for each $R$-module $M$. In particular,
 we have $\mathbf{a}=t(R)=tr_V(R)$ and, by applying lemma \ref{lem.quasi-tilting implies
 projective} to the identity $1_V:V\longrightarrow V$, we conclude that $V$ is a finitely generated
 projective module.

 We next prove assertions 1 and 2:

 \vspace*{0.3cm}

 1) If $Q'$ exists, then the complex $P^\bullet := \cdots \longrightarrow 0\longrightarrow Q'\stackrel{0}{\longrightarrow}P\longrightarrow
 0 \longrightarrow \cdots$, concentrated in degrees $-1$ and $0$, satisfies assertion 2
 of proposition \ref{prop.HKM} since we know that $P[0]\oplus
 \frac{R}{t(R)}[1]$ is a progenerator of $\mathcal{H}_\mathbf{t}$.

 Conversely, suppose that $\mathbf{t}$ is HKM and let $P^\bullet := \cdots \longrightarrow 0\longrightarrow Q\stackrel{d}{\longrightarrow}P'\longrightarrow
 0 \longrightarrow \cdots $ be an HKM complex whose associated torsion pair is
 $\mathbf{t}$. Then, by proposition \ref{prop.HKM}, we know that the
 complex

 \begin{center}
 $G:= \cdots \longrightarrow 0\longrightarrow t(\text{Ker (d)})\longrightarrow Q\stackrel{d}{\longrightarrow}P'\longrightarrow
 0\longrightarrow \cdots$,
 \end{center}
 concentrated in degrees $-2,-1,0$, is a progenerator of
 $\mathcal{H}_\mathbf{t}$. We then have that
 $\text{add}_{\mathcal{H}_\mathbf{t}}(G)=\text{add}_{\mathcal{H}_\mathbf{t}}(P[0]\oplus\frac{R}{t(R)}[1])$.
 In particular, we get that $V:=H^0(G)$ is a projective module and, hence,
 also $\text{Im}(d)$ is projective. It follows that, up to
 isomorphism in the category $\mathcal{C}(R)$, we can rewrite $G$ as

\begin{center}
$\cdots \longrightarrow 0\longrightarrow
t(Q')\stackrel{\begin{pmatrix} 0\\ \iota
\end{pmatrix}}{\longrightarrow}\text{Im}(d)\oplus
Q'\stackrel{\begin{pmatrix} 1 & 0\\ 0 & 0
\end{pmatrix}}{\longrightarrow}\text{Im}(d)\oplus H^0(G)\longrightarrow
0 \longrightarrow \cdots$,
\end{center}
where $\iota :t(Q')\hookrightarrow Q'$ is the inclusion. This in
turn implies that $P^\bullet$ is isomorphic in $\mathcal{C}(R)$ to
the complex

\begin{center}
$\cdots \longrightarrow 0\longrightarrow \text{Im}(d)\oplus
Q'\stackrel{\begin{pmatrix} 1 & 0\\ 0 & 0
\end{pmatrix}}{\longrightarrow}\text{Im}(d)\oplus H^0(G)\longrightarrow
0 \longrightarrow \cdots$.
\end{center}
The fact that
$\text{add}(V)=\text{add}(H^0(G))=\text{add}(H^0(P[0]\oplus\frac{R}{t(R)}[1])=\text{add}(P)$
and $V\in\mathcal{X}(P^\bullet )$ implies that
$\text{Hom}_R(Q',P)=0$. Moreover, we have
$\text{add}(Q'/t(Q'))=\text{add}(H^{-1}(G))=\text{add}(H^{-1}(P[0]\oplus\frac{R}{t(R)}[1])=\text{add}(R/t(R))$.

In order to show that, in general, the pair $\mathbf{t}$ need not be
HKM, we consider a field $K$, an infinite dimensional $K$-vector
space $P$ and view it as left module over $R=\text{End}_K(P)$. It is
well-known that $P$ is a faithful simple projective $R$-module, so
that $\mathcal{T}=\text{Add}({}_RP)=\text{Gen}({}_RP)$ is closed
under taking submodules and, hence, $\mathbf{t}$ is hereditary.
However the faithful condition of ${}_RP$ implies that each
projective $R$-module embeds in a direct product of copies of $P$.
Then it does not exists a finitely generated projective $R$-module
$Q'$ such that $\text{Hom}_R(Q',P)=0$ and
$\text{add}(Q'/t(Q'))=\text{add}(R/t(R))$. Hence $\mathbf{t}$ is not
HKM.

\vspace*{0.3cm}

2) $\mathbf{t}$ is the right constituent pair of a TTF triple if,
and only if, $\mathcal{T}=\text{Gen}(P)$ is closed under taking
products in $R-\text{Mod}$. But this is equivalent to saying that
each product of copies of $P$ is in $\text{Gen}(P)$. By \cite[Lemma,
Section 1]{CM}), this happens exactly when $P$ is finitely generated
over its endomorphism ring.
\end{proof}

 Recall that a ring is left
\emph{semihereditary} when its finitely generated left ideals are
projective.

\begin{exem} \label{exem.sum of stalks}
 Let $\mathbf{a}$ be an idempotent two-sided ideal of $R$, let
$(\mathcal{C},\mathcal{T},\mathcal{F})$ be the associated TTF triple
in $R-\text{Mod}$ and let $\mathbf{t}=(\mathcal{T},\mathcal{F})$ be
its right constituent torsion pair. The following assertions are
equivalent:

\begin{enumerate}
\item $\frac{R}{\mathbf{a}}[0]\oplus\frac{\mathbf{a}}{t(\mathbf{a})}[1]$ is a
progenerator of $\mathcal{H}_\mathbf{t}$;
\item $\mathcal{H}_\mathbf{t}$ has a progenerator of the form $V[0]\oplus
Y[1]$, with $V\in\mathcal{T}$ and $Y\in\mathcal{F}$;
\item $\mathbf{a}$ is finitely generated on the left and  $\text{Ext}_R^2(R/\mathbf{a},?)$ vanishes on $\mathcal{F}$.
\end{enumerate}
In particular, if $R$ is left semi-hereditary and $\mathbf{t}$ is
the right constituent pair of a TTF triple in $R-\text{Mod}$, then
$\mathcal{H}_\mathbf{t}$ is a module category if, and only if, the
associated idempotent ideal is finitely generated on the left.

\end{exem}
\begin{proof}
 $1)\Longrightarrow 2)$ is clear.

$2)\Longrightarrow 3)$ By lemma \ref{lem.description  de t}, we know
that $V$ is finitely presented and
$\mathcal{T}=\text{Gen}(V)\subseteq\text{Ker}(\text{Ext}_R^1(V,?))$.
But we also have that $\mathcal{T}=\{T\in R-\text{Mod}:$
$\mathbf{a}T=0\}\cong\frac{R}{\mathbf{a}}-\text{Mod}$. We then get
that $V$ is a finitely presented generator of
$\frac{R}{\mathbf{a}}-\text{Mod}$ such that
$\text{Ext}_{R/\mathbf{a}}^1(V,?)=0$. That is, $V$ a progenerator of
$\frac{R}{\mathbf{a}}-\text{Mod}$, which implies that
$\text{add}_{R-\text{Mod}}(R/\mathbf{a})=\text{add}_{R-\text{Mod}}(V)$.
Then $R/\mathbf{a}$ is a finitely presented left $R$-module and,
hence, $\mathbf{a}$ is finitely generated as a left ideal. The fact
that $\text{Ext}_R^2(R/\mathbf{a},?)$ vanishes on $\mathcal{F}$
follows from the fact that, by corollary \ref{cor.progenerator
sum-of-stalks hereditary}, we know that $\text{Ext}_R^2(V,?)$
vanishes on $\mathcal{F}$.

$3)\Longrightarrow 1)$ We take $V=\frac{R}{\mathbf{a}}$ and
$Y=\frac{\mathbf{a}}{t(\mathbf{a})}$. Then conditions 2.a, 2.b and
2.d of corollary \ref{cor.progenerator sum-of-stalks hereditary} hold since
$F/tr_Y(F)$ is in $\mathcal{T}$, for all $F\in\mathcal{F}$. We just
need to prove that $Y$ is a projective $R/t(R)$-module since it is
clearly in ${}^\perp\mathcal{T}=\mathcal{C}$. Let $0\rightarrow
K\hookrightarrow Q\stackrel{q}{\longrightarrow}\mathbf{a}\rightarrow
0$ be an exact sequence, with $Q$ a finitely generated projective
$R$-module. The canonical projection $K\twoheadrightarrow K/t(K)$
extends to $Q$ since
$\text{Ext}_R^1(\mathbf{a},\frac{K}{t(K)})\cong\text{Ext}_R^2(\frac{R}{\mathbf{a}},\frac{K}{t(K)})=0$.
It follows that the canonical monomorphism
$\iota:K/t(K)\longrightarrow Q/t(Q)$ splits. But its cokernel is
$\frac{Q}{K+t(Q)}\cong\frac{\frac{Q}{K}}{\frac{K+t(Q)}{K}}\cong\frac{\mathbf{a}}{q(t(Q))}$.
It follows that this latter one is a projective $R/t(R)$-module,
which implies that it is in $\mathcal{F}$ when viewed as an
$R$-module. But then
$\frac{t(\mathbf{a})}{q(t(Q))}\in\mathcal{T}\cap\mathcal{F}=0$.
Therefore we have $q (t(Q))=t(\mathbf{a})$ and
$\frac{\mathbf{a}}{t(\mathbf{a})}$ is projective as a left
$R/t(R)$-module.
\end{proof}

We are now able to give a second significative class of rings for
which we can identify all hereditary torsion pairs whose
heart is a module category.

\begin{prop} \label{prop.semihereditary}
Let $R$ be a left semihereditary ring and let $V$ be a finitely
presented quasi-tilting $R$-module whose associated torsion pair
$\mathbf{t}=(\text{Gen}(V),\text{Ker}(\text{Hom}_R(V,?)))$ is
hereditary. The following assertions are equivalent:

\begin{enumerate}
\item  If $\mathbf{a}=\text{ann}_R(V/t(R)V)$ then
$\mathbf{a}/t(R)$ is an idempotent ideal of $R/t(R)$, which is
finitely generated on the left,  and there is a monomorphism
$\xymatrix{R/\mathbf{a} \hspace{0.05cm}\ar@{^(->}[r]
&(V/t(R)V)^{(n)}}$, for some natural number $n$.

\item The heart $\mathcal{H}_\mathbf{t}$ is a module category.
\end{enumerate}
In this case $\mathcal{H}_\mathbf{t}$ is equivalent to
$S-\text{Mod}$, where
$S=\begin{pmatrix}\text{End}_R(\frac{\mathbf{a}}{t(R)})^{op} & 0\\
\text{Ext}_R^1(V,\frac{\mathbf{a}}{t(R)}) & \text{End}_R(V)^{op}
\end{pmatrix}$.
\end{prop}
\begin{proof}
$1)\Longrightarrow 2)$ Put $\bar{M}=M/t(R)M$, for each $R$-module
$M$. Note that
$\mathcal{T}\cap\bar{R}-\text{Mod}=\text{Gen}(\bar{V})$ and that
$\text{ann}_{\bar{R}}(\bar{V})=\bar{\mathbf{a}}$. We then get that
$\text{Hom}_R(\bar{\mathbf{a}},T)=0$, for each $T\in\mathcal{T}$.
Indeed if $\xymatrix{f:\bar{\mathbf{a}} \ar[r]& T}$ is any
$R$-homomorphism, then
$\text{Im}(f)\in\mathcal{T}\cap\bar{R}-\text{Mod}$ and the induced
morphism $\xymatrix{\bar{f}:\bar{\mathbf{a}}\ar[r]&\text{Im}(f)}$ is
a morphism in $\bar{R}-\text{Mod}$ such that
$\bar{f}(\bar{\mathbf{a}})=\bar{f}(\bar{\mathbf{a}}^2)=\bar{\mathbf{a}}\text{Im}(f)=0$.
We then have that $\bar{\mathbf{a}}$ is in ${}^\perp\mathcal{T}$.

On the other hand, since  $\bar{\mathbf{a}}$ is finitely generated
on the left, we have a finitely generated left ideal $\mathbf{a}'$
of $R$ contained in $\mathbf{a}$ such that the canonical composition
$\xymatrix{\mathbf{a}' \hspace{0.05cm}\ar@{^(->}[r] &\mathbf{a}
\ar@{>>}[r]&\bar{\mathbf{a}}}$ is an epimorphism.  We then get that
$\frac{\mathbf{a}'}{t(\mathbf{a}')}=\frac{\mathbf{a}'}{\mathbf{a}'\cap
t(R)}\cong\bar{\mathbf{a}}$ and, since $\mathbf{a}'$ is projective,
we conclude that $\bar{\mathbf{a}}$ is a finitely generated
projective left $\bar{R}$-module.

Note now that $V$ and $Y:=\bar{\mathbf{a}}$ satisfy all conditions
2.a-c of corollary \ref{cor.progenerator sum-of-stalks hereditary}.
Moreover if $F\in\mathcal{F}$ then $F/\mathbf{a}F$ is generated by
$R/\mathbf{a}$ and, due to our hypotheses, we know that
$F/\mathbf{a}F$ is in $\mathcal{T}$,  so that also property 2.d of
that corollary holds. Then $\mathcal{H}_\mathbf{t}$ is a module
category, actually equivalent to $S-\text{Mod}$ (see proposition
\ref{prop.progenerator as sum of stalks}).

$2)\Longrightarrow 1)$ Let $G$
be a complex as in theorem \ref{teor.hereditary case}, which is then
a progenerator of $\mathcal{H}_\mathbf{t}$. The fact that
$\text{Im}(d)$ is projective easily implies that $G$ is isomorphic
to $H^0(G)[0]\oplus H^{-1}(G)[1]$ in $\mathcal{H}_\mathbf{t}$. Putting $V=H^0(G)$ and
$Y=H^{-1}(G)$ for simplicity, corollary \ref{cor.progenerator
sum-of-stalks hereditary} and its proof show that
$\mathbf{t}'=(\mathcal{T}\cap\bar{R}-\text{Mod},\mathcal{F})$ is the
right constituent torsion pair of a TTF triple in
$\bar{R}-\text{Mod}$ defined by the idempotent ideal
$\bar{\mathbf{a}}=\text{ann}_{\bar{R}}(\bar{V})$, which is finitely
generated on the left. Then we have
$\bar{\mathbf{a}}=\mathbf{a}/t(R)$, where
$\mathbf{a}=\text{ann}_R(V/t(R)V)$. Moreover, we have
$\frac{R}{\mathbf{a}}-\text{Mod}=\mathcal{T}\cap\bar{R}-\text{Mod}=\text{Gen}(\bar{V})$,
so that $R/\mathbf{a}\in\text{add}(\bar{V})$.
\end{proof}

\section{When the torsion class is closed under taking products}

Our next result shows that if the torsion class is closed under
taking products in $R-\text{Mod}$, then classical tilting theory
appears quite naturally.

\begin{teor} \label{teor.closed under products case}
Let $\mathbf{t}=(\mathcal{T},\mathcal{F})$ be a torsion pair in
$R-\text{Mod}$. The following assertions are equivalent:

\begin{enumerate}
\item $\mathcal{T}$ is closed under taking products in
$R-\text{Mod}$ and the heart $\mathcal{H}_\mathbf{t}$ is a module
category;
\item  $\mathcal{T}=\text{Gen}(V)$, where $V$ is a module which is
classical 1-tilting over $R/\text{ann}_R(V)$ and  admits a finitely
generated projective presentation
$\xymatrix{Q\ar[r]^{d}&P\ar[r]&V\ar[r]& 0}$ in $R-\text{Mod}$ and a
submodule $X\subseteq\text{Ker}(d)$ such that:

\begin{enumerate}
\item $\frac{Ker(d)}{X}\in\mathcal{F}$ and $\text{Ker}(d)\subseteq
X+\mathbf{a}Q$, where $\mathbf{a}:=\text{ann}_R(V)$;
\item $\text{Ext}_R^1(Q/X,?)$ vanishes on
$\mathcal{F}$;
\item There is a $R$-homomorphism $\xymatrix{h:(\frac{Q}{X})^{(I)}\ar[r]& R/t(R)}$, for some
set $I$, such that
$h((\frac{\text{Ker}(d)}{X})^{(I)})=\frac{\mathbf{a}+t(R)}{t(R)}$.
\end{enumerate}
\end{enumerate}
In this case
$\mathbf{t}'=(\text{Gen}(V),\mathcal{F}\cap\frac{R}{\mathbf{a}}-\text{Mod})$
is a classical tilting torsion pair in $R/\mathbf{a}-\text{Mod}$ and
the forgetful functor
$\mathcal{H}_{\mathbf{t}'}\longrightarrow\mathcal{H}_\mathbf{t}$ is
faithful.
\end{teor}
\begin{proof}
$2)\Longrightarrow 1)$ The classical $1$-tilting condition of $V$
implies that $\mathcal{T}$ consists of the $R/\mathbf{a}$-modules
$T$ such that $\text{Ext}_{R/\mathbf{a}}^1(V,T)=0$. This class is
clearly closed under taking products. We next consider the complex

\begin{center}
$\xymatrix{G:= & \cdots \ar[r] & 0 \ar[r]& X
\hspace{0.1cm}\ar@{^(->}[r] & Q \ar[r]^{d} & P \ar[r] & 0 \ar[r] &
\cdots}$
\end{center}
concentrated in degrees $-2,-1,0$. By condition 2.a and by the
equality $\mathcal{T}=\text{Gen}(V)$, we have that
$G\in\mathcal{H}_\mathbf{t}$. We shall check that $G$ satisfies the standard conditions 1-5. 
We inmediatly derive the standard conditions 2, 3 and 4. 
On the other hand, our condition 2.c implies that
$\text{Coker}(h_{| (\frac{\text{Ker}(d)}{X})^{(I)}})$ is isomorphic
to $\frac{R}{\mathbf{a}+t(R)}$ and, hence, it is in
$\frac{R}{\mathbf{a}}-\text{Mod}$. But we have that
$\frac{R}{\mathbf{a}}-\text{Mod}=\overline{Gen}(V)$ since
$R/\mathbf{a}\in\overline{Gen}(V)$ due to the $1$-tilting condition
of $V$ over $R/\mathbf{a}$. Then the standard condition 5 
holds. It remains to prove
that $\mathcal{T}\subseteq\text{Ker}(\text{Ext}_R^1(V,?))$. Let

\begin{center} $\xymatrix{0\ar[r] &
T\ar[r] & M \ar[r] & V \ar[r] &0}$
\end{center}
be any exact sequence in $R-\text{Mod}$, with $T\in\mathcal{T}$.
Since $\mathcal{T}=\text{Gen}(V)$ is closed under taking extensions
in $R-\text{Mod}$, we get that $M\in\mathcal{T}$ and, hence, the
exact sequence  lives in $R/\mathbf{a}-\text{Mod}$. But then it
splits since we have that
$\mathcal{T}=\text{Ker}(\text{Ext}_{R/\mathbf{a}}^1(V,?))$.

$1)\Longrightarrow 2)$ Let $G$ be a complex in standard form satisfying the standard conditions 1-5, which is then a  
progenerator of $\mathcal{H}_\mathbf{t}$ (see theorem \ref{prop.module heart of torsion pair}),  and let us put
$V:=H^0(G)$ and $\mathbf{a}=\text{ann}_R(V)$. There is an obvious
monomorphism $\xymatrix{\frac{R}{\mathbf{a}} \ar[r] &V^V}$ and, by
hypothesis, we have that $V^V\in\mathcal{T}=\text{Gen}(V)$.  We then
get that $\overline{Gen}(V)=\frac{R}{\mathbf{a}}-\text{Mod}$.
Moreover, from lemma \ref{lem.description de t} and
\cite[Proposition 3.2]{CMT} we get that $V$ is a classical
$1$-tilting $R/\mathbf{a}$-module. On the other hand, the equality
$\overline{Gen}(V)=\frac{R}{\mathbf{a}}-\text{Mod}$ implies that
$\text{Rej}_\mathcal{T}(M)=\text{Rej}_{\frac{R}{\mathbf{a}}-\text{Mod}}(M)=\mathbf{a}M$.
Then conditions 2.a and 2.b follow directly.

It just remains to check condition 2.c. To do that, consider the
morphism $\xymatrix{h:(\frac{Q}{X})^{(I)} \ar[r] &\frac{R}{t(R)}}$
in the standard condition 5 
and put $h':=h_{| (\frac{\text{Ker}(d)}{X})^{(I)}}$. The fact that
$\text{Coker}(h')$ is in
$\overline{Gen}(V)=\frac{R}{\mathbf{a}}-\text{Mod}$ is equivalent to
saying that $\mathbf{a}\frac{R}{t(R)}\subseteq\text{Im}(h')$. But,
by the already proved condition 2.a, we know that
$\text{Im}(h')\subseteq
h(\mathbf{a}(\frac{Q}{X})^{(I)})=\mathbf{a}\text{Im}(h)\subseteq
\mathbf{a}\frac{R}{t(R)}$. We then get that
$\text{Im}(h')=\mathbf{a}\frac{R}{t(R)}=\frac{\mathbf{a}+t(R)}{t(R)}$.

Finally, it is clear that
$\mathbf{t}'=(\mathcal{T},\mathcal{F}\cap\frac{R}{\mathbf{a}}-\text{Mod})$
is a classical tilting torsion pair of $R/\mathbf{a}-\text{Mod}$. If
$\xymatrix{j:\mathcal{H}_{\mathbf{t'}}\ar[r]
&\mathcal{H}_{\mathbf{t}}}$ is the forgetful functor then, arguing
as in the final part of the proof of corollary \ref{cor.progenerator
sum-of-stalks hereditary}, in order to prove that $j$ is faithful,
we just need to check that the canonical map \linebreak
$\xymatrix{\text{Hom}_{\mathcal{H}_{\mathbf{t'}}}(V[0],M)\ar[r]
&\text{Hom}_{\mathcal{H}_{\mathbf{t}}}(V[0],M)}$ is injective, for
all $M\in\mathcal{H}_{\mathbf{t'}}$. Similar as there, this in turn
reduces to check that the canonical map

\begin{center}
$\xymatrix{\text{Ext}_{R/\mathbf{a}}^1(V,F)\cong\text{Hom}_{\mathcal{H}_{\mathbf{t'}}}(V[0],F[1])\ar[r]&
\text{Hom}_{\mathcal{H}_{\mathbf{t}}}(V[0],F[1])\cong\text{Ext}_{R}^1(V,F)}$
\end{center}
is injective, for all
$F\in\mathcal{F}\cap\frac{R}{\mathbf{a}}-\text{Mod}$. But this is
clear.
\end{proof}

Note that if $V$ is a non-projective classical 1-tilting $R$-module,
then  (see \cite{Mi})  $V$ is also a classical tilting right
$S$-module, where $S=\text{End}({}_RV)^{op}$, such that the
canonical algebra morphism $R\longrightarrow\text{End}(V_S)$ is an
isomorphism. Due to the tilting theorem, we then  know that
$(\text{Ker}(?\otimes_AV),\text{Ker}(\text{Tor}_1^A(?,V)))$ is a
torsion pair in $\text{Mod}-A$. If we had
$\text{Ker}(?\otimes_AV)=0$ we would have that
$\text{Tor}_1^R(?,V)=0$, and hence $V$ would be a flat left
$A$-module, which is a contradiction (see \cite[Corollaire 1.3]{L}).
Then there is a right $R$-module $X\neq 0$  such that
$X\otimes_AV=0\neq \text{Tor}_1^A(X,V)$. Considering an epimorphism
$X\twoheadrightarrow X'$, with $X'_A$ simple,  and replacing $X$ by
$X'$, we can even choose $X$ to be a simple right $A$-module.

 Recall that if $A$ is
a ring and $M$ is an $A$-bimodule, then the \emph{trivial extension
of $A$ by $M$}, denoted $A\rtimes M$, is the ring whose underlying
$A$-bimodule is $A\oplus M$ and the multiplication is given by
$(a,m)\cdot (a',m')=(aa',am'+ma')$.

We can now give a systematic way of constructing   negative answers
to question \ref{ques.classical tilting pair}.

\begin{teor} \label{teor.nontilting pair with stalk progenerator}
Let $A$ be a finite dimensional algebra over an algebraically closed
field $K$, let $V$ be a classical 1-tilting left $A$-module such
that $\text{Hom}_A(V,A)=0$, let $X$ be a simple right $A$-module
such that $X\otimes_AV=0$ and let us consider the trivial extension
$R=A\rtimes M$, where $M=V\otimes_KX$. Viewing $V$ as a left
$R$-module annihilated by $0\rtimes M$, the pair
$\mathbf{t}=(\text{Gen}(V),\text{Ker}(\text{Hom}_R(V,?)))$ is a
non-tilting torsion pair in $R-\text{Mod}$ such that $V[0]$ is a
progenerator of $\mathcal{H}_\mathbf{t}$.
\end{teor}
\begin{proof}
All throughout the proof, for any two-sided ideal $\mathbf{a}$ of a
ring $R$, we view $R/\mathbf{a}$-modules as $R$-modules annihilated
by $\mathbf{a}$. Note that if $M$ is any such module and we apply
$?\otimes _RM:\text{Mod}-R\longrightarrow \text{Ab}$ to the
canonical  exact sequence $0\rightarrow\mathbf{a}\hookrightarrow
R\longrightarrow R/\mathbf{a}\rightarrow 0$, then we get an
isomorphism
$R\otimes_RM\stackrel{\cong}{\longrightarrow}\frac{R}{\mathbf{a}}\otimes_RM$,
which implies that the canonical morphism
$\text{Tor}_1^R(\frac{R}{\mathbf{a}},M)\longrightarrow\mathbf{a}\otimes_RM$
is an isomorphism.

 Bearing in mind that $X$ is simple in $\text{Mod}-A$, if $0\rightarrow Q'\stackrel{u}{\longrightarrow}
P'\longrightarrow V\rightarrow 0$ is the minimal projective
resolution of ${}_AV$, then the map $1_X\otimes
u:X\otimes_AQ'\longrightarrow X\otimes_AP'$ is the zero map. This
implies  that we have isomorphisms of vector spaces
$X\otimes_AP'\cong X\otimes_AV=0$ and $\text{Tor}_1^A(X,V)\cong
X\otimes_AQ'$.

Note that we have an isomorphism ${}_AM\cong {}_AV$ in
$A-\text{Mod}$ because, due to the algebraically closed condition of
$K$, the simple right $A$-module $X$ is one-dimensional over $K$. As
a right $A$-module, $M_A$ is in $\text{add}(X_A)$. Moreover, we have
$M\otimes_AV\cong V\otimes_KX\otimes_AV=0$.   We shall prove the
following facts:

\begin{enumerate}
\item[i)] $\text{ann}_R(V)=0\rtimes M=:\mathbf{a}$ and this ideal is the trace of $V$ in $R$;
\item[ii)]
$\mathbf{a}\otimes_RV=0$;
\item[iii)]  $\mathcal{T}:=\text{Gen}(V)$  is
closed under taking extensions, and hence a torsion class,  in
$R-\text{Mod}$;
\item[iv)] $\text{Ker}(\text{Hom}_A
(V,?))=\text{Ker}(\text{Hom}_R(V,?))$ and, hence,
$\mathbf{t}:=(\text{Gen}(V),\text{Ker}(\text{Hom}_R(V,?)))$ is a
torsion pair in $R-\text{Mod}$ which lives in $A-\text{Mod}$;
\item[v)] There is a finitely generated projective presentation of ${}_RV$

\begin{center}
$Q\stackrel{d}{\longrightarrow}P\twoheadrightarrow V\rightarrow 0$
\end{center}
such that $\text{Ker}(d)$ is a non-projective $R$-module in
$\mathcal{T}$.
\end{enumerate}

Suppose that all these facts have been proved. Then $\mathbf{t}$ is
a torsion pair in $R-\text{Mod}$ whose torsion class is closed under
taking products. We claim that  $V$ satisfies all conditions of
assertion 2 in theorem \ref{teor.closed under products case}, by
taking $X=\text{Ker}(d)$. The only nontrivial things to check are
conditions 2.b and 2.c in that assertion. Condition 2.b follows by
applying the exact sequence of $\text{Ext}_R(?,F)$, with
$F\in\mathcal{F}$, to the short exact sequence
$0\rightarrow\text{Ker}(d)\hookrightarrow
Q\stackrel{d}{\longrightarrow}\text{Imd}(d)\rightarrow 0$. On the
other hand, $t(R)$ is the trace of $V$ in $R$ and, by fact i), we
know that $t(R)=\mathbf{a}$. Then condition 2.c of assertion 2 in
theorem \ref{teor.closed under products case} also holds, simply by
taking as $h$ the zero map. Looking at the proof of implication
$2)\Longrightarrow 1)$ in that theorem, we see that the complex

\begin{center}
$G:=\cdots \longrightarrow
0\longrightarrow\text{Ker}(d)\hookrightarrow
Q\stackrel{d}{\longrightarrow}P\longrightarrow 0 \longrightarrow
\cdots$,
\end{center}
concentrated in degrees $-2,-1,0$, is a progenerator of
$\mathcal{H}_\mathbf{t}$. But we have an isomorphism $G\cong V[0]$
in $\mathcal{H}_\mathbf{t}$. Therefore $V[0]$ is a progenerator of
$\mathcal{H}_\mathbf{t}$. Note that this torsion pair $\mathbf{t}$
in $R-\text{Mod}$ is not  tilting,  because the projective dimension
of ${}_RV$ is $>1$ (see \cite[Section 2]{C}).

We now pass to prove the facts i)-v) in the list above:

\vspace*{0.3cm}

i) By definition of the $R$-module structure on $V$, we have
$(a,m)v=av$, for each $(a,m)\in A\rtimes M$ and each $v\in V$. Then
$\text{ann}_R(V)=\text{ann}_A(V)\rtimes M$.  But $\text{ann}_A(V)=0$
due to the $1$-tilting condition of ${}_AV$. On the other hand,  if
$f:V\longrightarrow R$ is a morphism in $R-\text{Mod}$, then we have
two $K$-linear maps $g:V\longrightarrow A$ and $h:V\longrightarrow
M$ such that $f(v)=(g(v),h(v))\in A\rtimes M=R$, for all $v\in V$.
Direct computation shows that $g$ is a morphism in $A-\text{Mod}$,
and hence $g=0$. Then one immediately sees that
$h\in\text{Hom}_A(V,M)$ and, since $_{A}M$ is in
$\text{Gen}({}_AV)$, we conclude that $t(R)=0\rtimes M=\mathbf{a}$.

\vspace*{0.3cm}

ii) Since $\mathbf{a}^2=0$ and $\mathbf{a}V=0$, we have an equality
$\mathbf{a}\otimes_RV=\mathbf{a}\otimes_{\frac{R}{\mathbf{a}}}V=\mathbf{a}\otimes_AV$.
But, as a right $A$-module,  we have that $\mathbf{a}_A\cong M_A$.
It follows that $\mathbf{a}\otimes_AV\cong M\otimes_AV=0$.

\vspace*{0.3cm}

iii)  Consider  any exact sequence $0\rightarrow T\longrightarrow
N\longrightarrow T'\rightarrow 0$ (*) in $R-\text{Mod}$, with
$T,T'\in\text{Gen}(V)=\mathcal{T}$. Taking an epimorphism
$p:V^{(I)}\twoheadrightarrow T'$ and taking the image of the last
exact sequence by the morphism
$\text{Ext}_R^1(p,T):\text{Ext}_R^1(T',T)\longrightarrow\text{Ext}_R^1(V^{(I)},T)$,
we easily reduce the problem to the case when $T'=V^{(I)}$.  We now
apply  the functor
$\frac{R}{\mathbf{a}}\otimes_R?:R-\text{Mod}\longrightarrow
\frac{R}{\mathbf{a}}-\text{Mod}$ to the sequence (*), with
$T'=V^{(I)}$, and use the fact that, by the initial paragraph of
this proof, we have
$\text{Tor}_1^R(R/\mathbf{a},V^{(I)})\cong\mathbf{a}\otimes_RV^{(I)}=0$.
We then get a commutative diagram with exact rows, where the
vertical arrows are the canonical maps:

$$\xymatrix{0 \ar[r] & T \ar[r] \ar[d]^{\wr} & N \ar[r] \ar[d] & V^{(I)} \ar[r] \ar[d]^{\wr} & 0  \\ 0 \ar[r] & \frac{R}{\mathbf{a}}\otimes_R T \ar[r] & \frac{R}{\mathbf{a}}\otimes_R N \ar[r] & \frac{R}{\mathbf{a}}\otimes_R V^{(I)} \ar[r] & 0 }$$

The central vertical arrow $N\longrightarrow
\frac{R}{\mathbf{a}}\otimes_RN\cong\frac{N}{\mathbf{a}N}$ is then an
isomorphism. This implies that $\mathbf{a}N=0$ and, hence, the
sequence (*) above lives in $A-\text{Mod}$ and $N\in\mathcal{T}$.

\vspace*{0.3cm}

iv) If $F\in\text{Ker}(\text{Hom}_R(V,?))$ then $t(R)F=tr_V(R)F=0$.
By fact i),  we get that $\mathbf{a}F=0$. Then $F$ is an $A$-module,
and hence
$\text{Ker}(\text{Hom}_R(V,?))\subseteq\text{Ker}(\text{Hom}_A(V,?))$.
The converse inclusion is obvious.

\vspace*{0.3cm}

v) The multiplication map $\mu :R\otimes_AV\longrightarrow V$ is
surjective. Moreover, we have an isomorphism of $K$-vector spaces

\begin{center}
$R\otimes_AV\cong (A\oplus M)\otimes_AV\cong V\oplus
(M\otimes_AV)=V\oplus 0\cong V$.
\end{center}
Since $V$ is a finite dimensional $K$-vector space we get that $\mu$
is an isomorphism of left $R$-modules.

Let $0\rightarrow Q'\stackrel{d'}{\longrightarrow}P'\longrightarrow
V\rightarrow 0$ be a finitely generated projective presentation of
$V$ in $A-\text{Mod}$. Using the previous paragraph, we then get a
finitely generated projective presentation of $V$ in $R-\text{Mod}$:

\begin{center}
$R\otimes_AQ'\stackrel{1\otimes d'}{\longrightarrow}
R\otimes_AP'\longrightarrow V\rightarrow 0$.
\end{center}
Then we have isomorphisms of $K$-vector spaces

\begin{center}
 $\text{Ker}(1\otimes
d')=\text{Tor}_1^A(R,V)\cong \text{Tor}_1^A(A\oplus
M,V)\cong\text{Tor}_1^A(M,V)$.
\end{center}
It is easy to deduce from this that $\mathbf{a}\text{Ker}(1\otimes
d')=0$. Then we can view $\text{Ker}(1\otimes d')$ as a left
$A$-module isomorphic to $\text{Tor}_1^A(M,V)$. Since
$M=V\otimes_KX$ we get that $\text{Tor}_1^A(M,V)\cong
V\otimes_K\text{Tor}_1^A(X,V)$ which is nonzero and isomorphic to
$V\otimes_K (X\otimes_AQ')\in\text{add}({}_AV)$ due to the first
paragraph of this proof.   It follows that $\text{Ker}(1\otimes d')$
is a nonzero left $R$-module in
$\text{add}({}_RV)=\text{add}({}_AV)$.

We finally prove that $W:=\text{Ker}(1\otimes d')$ is not a
projective in $R-\text{Mod}$. If it were so, we would have that
$W=tr_V(R)W$. By fact i), we would get that $\mathbf{a}W=W$, which
would imply that $W=0$ since $\mathbf{a}^2=0$.
\end{proof}

\section{Torsion pairs which are right constituents of TTF triples}

As shown in theorem \ref{teor.hereditary case} and corollaries
\ref{cor.faithful hereditary pair} and \ref{cor.progenerator
sum-of-stalks hereditary}, hereditary torsion pairs which are the
right constituent of a TTF triple appear quite naturally when
studying the modular condition of the heart. In this section we fix
an idempotent ideal $\mathbf{a}$ of $R$ and its associated TTF
triple $(\mathcal{C},\mathcal{T},\mathcal{F})$ and want to study
when the pair $\mathbf{t}=(\mathcal{T},\mathcal{F})$ has the
property that its heart $\mathcal{H}_\mathbf{t}$ is a module
category. When this is the case, by theorem \ref{teor.hereditary
case}, we know that $\mathbf{a}$ is finitely generated on the left.

 Our next result, very important in the sequel,
shows that the conditions for $\mathcal{H}_\mathbf{t}$ to be a
module category get rather simplified if we assume that the monoid
morphism $\xymatrix{V(R)\ar[r] & V(R/\mathbf{a})}$ (see section
\ref{sec.Terminology}) is an epimorphism.

\begin{teor} \label{teor.modular H with easy conditions}
Let $\mathbf{a}$ be an idempotent  ideal of the ring $R$ which is
finitely generated on the left, and let
$(\mathcal{C},\mathcal{T},\mathcal{F})$ be the associated TTF
triple. Consider the following assertions for
$\mathbf{t}=(\mathcal{T},\mathcal{F})$:

\begin{enumerate}
\item There is a finitely generated projective $R$-module $P$ satisfying the
following conditions:

\begin{enumerate}
\item $P/\mathbf{a}P$ is a (pro)generator of $R/\mathbf{a}-\text{Mod}$;
\item There is an exact sequence $\xymatrix{0\ar[r] & F\ar[r] & C\ar[r]^{q}&\mathbf{a}P\ar[r] & 0}$ in $R-\text{Mod}$,
where  $F\in\mathcal{F}$ and
 $C$ is a finitely generated module which is in $\mathcal{C}\cap
\text{Ker}(\text{Ext}_R^1(?,\mathcal{F}))$ and generates
$\mathcal{C}\cap\mathcal{F}$.
\end{enumerate}
\item The heart $\mathcal{H}_\mathbf{t}$ is a module category.
\end{enumerate}
Then 1) implies 2) and, in such a case, if $\xymatrix{j:\mathbf{a}P
\hspace{0.07cm}\ar@{^(->}[r]& P}$ is the inclusion, then the complex
concentrated in degrees $-1,0$

\begin{center}
$\xymatrix{G':= \cdots \ar[r] &  0 \ar[r] &
C\oplus\frac{C}{t(C)}\ar[rr]^{\hspace{0.3cm}\begin{pmatrix}j\hspace{0.02cm}q
& 0
\end{pmatrix}} &&P \ar[r] &  0 \ar[r] &  \cdots}$
\end{center}
is a progenerator of $\mathcal{H}_\mathbf{t}$.

When the monoid morphism $\xymatrix{V(R)\ar[r]& V(R/\mathbf{a})}$ is
surjective, the implication $2)\Longrightarrow 1)$ is also true.
\end{teor}
\begin{proof}
$1)\Longrightarrow 2)$ Fix an exact sequence $\xymatrix{0\ar[r]
&F\ar[r]&C\ar[r]^{q} &\mathbf{a}P\ar[r]& 0}$ as indicated in
condition 1.b. Taking $F'\in\mathcal{F}$ arbitrary,  applying the
exact sequence of $\text{Ext}_{R}(?,F')$ to the sequence \linebreak
$\xymatrix{0\ar[r]& t(C)\hspace{0.07cm}
\ar@{^(->}[r]&C\ar@{>>}[r]^{pr. \hspace{0.35cm}}&C/t(C)\ar[r]&0}$
and using condition 1.b, we get that
$\text{Ext}_R^1(C/t(C),?)_{|\mathcal{F}}=0$. But then any
epimorphism $\xymatrix{(R/t(R))^{n}\ar@{>>}[r] &
 \frac{C}{t(C)}}$ splits, which implies that
 $U:=C/t(C)$ is a finitely generated projective $R/t(R)$-module which is in
 $\mathcal{C}$. Moreover it generates
 $\mathcal{C}\cap\mathcal{F}$ since so does $C$.

Let $\xymatrix{\pi_U:Q_U\ar@{>>}[r]& U}$ and
$\xymatrix{\pi_C:Q_C\ar@{>>}[r]&C}$ be two epimorphisms from
finitely generated projective modules, whose respective kernels are
denoted by $K_U$ and $K_C$. We will prove that the following complex in standard form,  which is clearly
quasi-isomorphic to $G'$, is a progenerator of
$\mathcal{H}_\mathbf{t}$.

\begin{center}
$G:$ \hspace*{1cm} $\xymatrix{\cdots \ar[r] & 0 \ar[r] & K_U\oplus K_C \hspace{0.08cm}
\ar@{^(->}[r] &Q_U\oplus Q_C \ar[rr]^{\hspace{0.5cm}\begin{pmatrix}0
& jq\pi_C\end{pmatrix}}&&P \ar[r] & 0
 \ar[r] & \cdots.}$
\end{center}

We have that $H^0(G)=P/\mathbf{a}P\in\mathcal{T}$ and, by an
appropriate use of ker-coker lemma, we get an exact sequence
$\xymatrix{0\ar[r]& K_C
\hspace{0.07cm}\ar@{^(->}[r]&\text{Ker}(q\pi_C)\ar[r]& F\ar[r]&0}$.
It then follows that

\begin{center}
$H^{-1}(G)=\frac{\text{Ker}(\begin{pmatrix}0 &
jq\pi_C\end{pmatrix})}{K_U\oplus
K_C}=\frac{Q_U\oplus\text{Ker}(q\pi_C)}{K_U\oplus K_C}\cong U\oplus
F$,
\end{center}
which is in $\mathcal{F}$.  This proves that $G$ is an object of
$\mathcal{H}_\mathbf{t}$. We next check all conditions 3.a-d of
theorem \ref{teor.hereditary case}. Clearly, condition 3.a in that
theorem is a consequence of our condition 1.a.  Note next that
$(K_U\oplus K_C)+\mathbf{a}(Q_U\oplus Q_C)=Q_U\oplus Q_C$ because
$U\cong\frac{Q_U}{K_U}$ and $C\cong\frac{Q_C}{K_C}$ are both in
$\mathcal{C}=\{X\in R-\text{Mod}:$ $\mathbf{a}X=X\}$. Then condition
3.b of the mentioned theorem is automatic, as so is condition 3.c
since $\frac{Q_U\oplus Q_C}{K_U\oplus K_C}\cong U\oplus C$.

Put now $X=K_U\oplus K_C$ and $Q=Q_U\oplus Q_C$ as in the standard notation. Using the fact that $U$ generates $\mathcal{C}\cap\mathcal{F}$, fix
an epimorphism
$\xymatrix{p:U^{(J)}\ar[r]&\mathbf{a}/t(\mathbf{a})}$. Identifying
$Q/X=U\oplus C$, we clearly have that $\xymatrix{(p
\hspace{0.3cm} 0 ):(Q/X)^{(J)}=U^{(J)}\oplus C^{(J)}
\ar[r]& \mathbf{a}/t(\mathbf{a})}$ is a homomorphism whose restriction to
$H^{-1}(G)^{(J)}=U^{(J)}\oplus F^{(J)}$ is  an epimorphism. Then
also condition 3.d of Theorem \ref{teor.hereditary case} holds.

\vspace*{0.3cm}

$2)\Longrightarrow 1)$ (assuming that the monoid morphism
$\xymatrix{V(R)\ar[r] & V(R/\mathbf{a})}$ is surjective). Let $G$ be
a progenerator of $\mathcal{H}_\mathbf{t}$. By theorem
\ref{teor.hereditary case}, we know that $\text{add}(H^0(G))=
\text{add}(R/\mathbf{a})$ and our extra hypothesis gives a finitely
generated projective $R$-module $P$ such that $P/\mathbf{a}P\cong
H^0(G)$, so that condition 1.a holds. Fixing such a $P$ and
following the proof of theorem \ref{prop.module heart of torsion
pair}, we see that we can represent $G$ by a chain complex

\begin{center}
$\xymatrix{\cdots \ar[r] & 0 \ar[r] & X \ar[r]^{j} &Q \ar[r]^{d} & P
\ar[r] & 0 \ar[r] &\cdots}$
\end{center}
 where $Q$ is finitely generated projective, $j$ is a monomorphism,
 $\text{Im}(d)=\mathbf{a}P$. Then $G$ satisfies all conditions 3.a-d
 of theorem \ref{teor.hereditary case}.
Note that
 $\mathbf{a}P=\mathbf{a}^2P=\mathbf{a}\text{Im}(d)=d(\mathbf{a}Q)$,
 which implies that $Q=\text{Ker}(d)+\mathbf{a}Q$ and, by condition 3.b of the mentioned theorem, that $Q= X+\mathbf{a}Q$. That is, the
 module $Q/X$ is in $\mathcal{C}$ and, by condition 3.c of that theorem, we
 also have that
 $Q/X\in\text{Ker}(\text{Ext}_R^1(?,\mathcal{F}))$. The exact
 sequence needed for our condition 1.b is then $\xymatrix{0\ar[r] & H^{-1}(G)\hspace{0.07cm}\ar@{^(->}[r] & Q/X\ar[r]^{\bar{d}}&\mathbf{a}P \ar[r] &
 0}$.
\end{proof}

We now give some applications of last theorem.

\begin{cor} \label{cor.torsion theory given by projective}
Let $Q$ be a finitely generated $R$-module and let us consider the
hereditary torsion pair $\mathbf{t}=(\mathcal{T},\mathcal{F})$,
where $\mathcal{T}=\text{Ker}(\text{Hom}_R(Q,?))$. If the trace  of
$Q$ in $R$ is finitely generated on the left, then $\mathbf{t}$ is
an HKM torsion pair and
 $\mathcal{H}_\mathbf{t}$ is a module category.
\end{cor}
\begin{proof}
 We will check assertion 1 of theorem \ref{teor.modular H with easy
conditions} for the  suitable choices. We have that $\mathcal{T}$
fits into a TTF triple $(\mathcal{C},\mathcal{T},\mathcal{F})$,
where $\mathcal{C}=\text{Gen}(Q)=\{T\in R-\text{Mod}:$
$\mathbf{a}T=T\}$, where $\mathbf{a}=tr_Q(R)$ (see \cite[Proposition
VI.9.4 and Corollary VI.9.5]{S}). Taking $P=R$ in theorem
\ref{teor.modular H with easy conditions}, we have an obvious
epimorphism $p:Q^n\twoheadrightarrow\mathbf{a}$, for some integer
$n>0$. We then take $C=\frac{Q^n}{t(\text{Ker}(p))}$ and
$q:C=\frac{Q^n}{t(\text{Ker}(p))}\twoheadrightarrow\mathbf{a}$ the
epimorphism defined by $p$. The fact that
$\text{Ext}_R^1(C,?)_{|\mathcal{F}}$ follows by taking
$F\in\mathcal{F}$ and applying the long exact sequence of
$\text{Ext}_R(?,F)$ to the short exact sequence $0\rightarrow
t(\text{Ker}(p))\hookrightarrow Q^n\twoheadrightarrow C\rightarrow
0$.

On the other hand, the progenerator $G$ of $\mathcal{H}_\mathbf{t}$
given in theorem \ref{teor.modular H with easy conditions} is
quasi-isomorphic to the complex

\begin{center}
$\cdots \longrightarrow 0\longrightarrow t(\text{Ker}(p))\oplus
t(Q)^{(n)}\hookrightarrow Q^{(n)}\oplus
Q^{(n)}\stackrel{\begin{pmatrix} jp & 0
\end{pmatrix}}{\longrightarrow} R\longrightarrow 0 \longrightarrow \cdots $,
\end{center}
where $j:\mathbf{a}\hookrightarrow R$ is the inclusion.  One readily
proves now that the complex $P^\cdot :$ $\cdots \longrightarrow
0\longrightarrow Q^{(n)}\oplus Q^{(n)}\stackrel{\begin{pmatrix} jp &
0
\end{pmatrix}}{\longrightarrow} R\longrightarrow 0 \longrightarrow \cdots$ satisfies
condition 2 of proposition \ref{prop.HKM}.
\end{proof}

 The easy proof of the
following auxiliary result is left to the reader.

\begin{lemma} \label{lem.idemptent ideals in semiperfect}
Let $R$ be a  ring and $\mathbf{a}$ be an idempotent ideal. The
following assertions hold:

\begin{enumerate}
\item If $\xymatrix{p:P\ar@{>>}[r] & M}$ is a projective cover and
$\mathbf{a}M=M$, then $\mathbf{a}P=P$;
\item Suppose that $R$ is semiperfect and let $\{e_1,\dots,e_n\}$ be a family of primitive
orthogonal idempotents such that $\sum_{1\leq i\leq n}e_i=1$.  If
$\mathbf{a}$ is finitely generated on the left, then there is an
idempotent element $e\in R$ (which is a sum of $e_i$'s) such that
$\mathbf{a}=ReR$.
\end{enumerate}
\end{lemma}

For semiperfect rings, we have the following result.

\begin{cor} \label{cor.semiperfect}
Let $R$ be a semiperfect ring, let $\{e_1,\dots,e_n\}$ be a complete
family of primitive orthogonal idempotents,  and let
$\mathbf{t}=(\mathcal{T},\mathcal{F})$ be the right constituent
torsion pair of a TTF triple in $R-\text{Mod}$. The following
assertions are equivalent:

\begin{enumerate}
\item The heart $\mathcal{H}_\mathbf{t}$ is a module category;
\item $\mathbf{t}$ is an HKM torsion pair;
\item There is an idempotent element $e\in R$ (which is a sum of $e_i$'s) such that  $ReR$ is finitely generated on the left
and $ReR$ is the idempotent ideal which defines the TTF triple.
\end{enumerate}
\end{cor}
\begin{proof}
$1)\Longrightarrow 3)$ By theorem \ref{teor.hereditary case}, we
know that $\mathbf{a}$ is finitely generated on the left. Then
assertion 3 follows from  lemma \ref{lem.idemptent ideals in
semiperfect}.

$3)\Longrightarrow 2)$ is a particular case of corollary
\ref{cor.torsion theory given by projective}.

$2)\Longrightarrow 1)$ follows from \cite[Theorem 3.8]{HKM}.

\end{proof}

As a consequence of theorem \ref{teor.hereditary case} and corollary
\ref{cor.semiperfect}, we get more significative classes of rings
for which we can identify all the hereditary torsion pairs whose
heart is a module category.

\begin{cor} \label{cor.local and artinian}
Let $\mathbf{t}=(\mathcal{T},\mathcal{F})$ be a hereditary torsion
pair in $R-\text{Mod}$ and let $\mathcal{H}_\mathbf{t}$ be its heart. The
following assertions hold:

\begin{enumerate}
\item If $R$ is a local ring and $\mathcal{H}_\mathbf{t}$ is a module category, then $\mathbf{t}$ is either
$(R-\text{Mod},0)$ or $(0,R-\text{Mod})$;
\item When $R$ is right perfect,   $\mathcal{H}_\mathbf{t}$ is a module
category if, and only if, there is an idempotent element $e\in R$
such that $\mathcal{T}=\{T\in R-\text{Mod}:$ $eT=0\}$ and $ReR$ is
finitely generated on the left;
\item If $R$ is left Artinian (e.g. an Artin algebra), then $\mathcal{H}_\mathbf{t}$ is
always module category.
\end{enumerate}
\end{cor}
\begin{proof}
1) By theorem \ref{prop.module heart of torsion pair}, we have a
finitely presented $R$-module $V$ such that
$\mathcal{T}=\text{Gen}(V)\subseteq\text{Ker}(\text{Ext}_R^1(V,?))$.
Using theorem \ref{teor.hereditary case} and its proof, we get that
$\mathbf{t}'=(\mathcal{T}\cap\frac{R}{t(R)}-\text{Mod},\mathcal{F})$
is the right constintuent of a TTF triple in $\bar{R}-\text{Mod}$
defined by an idempotent ideal $\bar{\mathbf{b}}=\mathbf{b}/t(R)$ of
$\bar{R}:=R/t(R)$ which is finitely generated on the left, where
$\mathbf{b}=ann_{R}(V/t(R)V)$. Since $\bar{R}$ is also a local ring,
and hence semiperfect, lemma \ref{lem.idemptent ideals in
semiperfect} says that $\bar{\mathbf{b}}=\bar{R}\bar{e}\bar{R}$, for
some idempotent element $\bar{e}\in\bar{R}$, which is necessarily
equal to $\bar{1}$ or $0$. The fact that $\bar{R}\in\mathcal{F}$
implies that $\bar{e}=1$, so that $\bar{b}=\bar{R}$ and
$b=R=ann_{R}(V/t(R)V)$, thus $V=t(R)V$ and, by lemma
\ref{lem.quasi-tilting implies projective} , we deduce that $V$ is
projective. But all finitely generated projective modules over a
local ring are free. Then we have $V=0$ or $V=R^{(n)},$ for some set $n\in\mathbb{N}$,  so that
either $\mathbf{t}=(0,R-\text{Mod})$ or $\mathbf{t}=
(R-\text{Mod},0).$

2)  Assume now that $R$ is right perfect and that
$\mathcal{H}_\mathbf{t}$ is a module category. By
  \cite[Corollary VIII.6.3]{S}, we know that $\mathbf{t}$ is the
  right constituent of a TTF triple. By theorem \ref{teor.hereditary
  case}, the associated idempotent ideal is finitely generated on
  the left and, by \cite[Corollary
  VIII.6.4]{S}, we know that it is  of
  the form $AeA$. Conversely, if $e\in A$ is idempotent and $AeA$ is
  finitely generated on the left and $\mathcal{T}=\{T\in R-\text{Mod}:$ $eT=0\}$, then corollary \ref{cor.semiperfect} says that
   $\mathcal{H}_\mathbf{t}$ is
  a module category.

3) This assertion is a direct consequence of \cite[Example
VI.8.2]{S} and corollary \ref{cor.semiperfect} since all left ideals
are finitely generated.
\end{proof}

Another consequence of theorem \ref{teor.modular H with easy
conditions} is the following.

\begin{cor} \label{cor.when a is torsionfree}
Let $\mathbf{a}$ be an idempotent ideal of $R$,  which is finitely
generated on the left, let $\mathbf{t}=(\mathcal{T},\mathcal{F})$ be
the right constituent torsion pair of the associated TTF triple in
$R$-Mod and suppose that $\mathbf{a}\in\mathcal{F}$ and that
 the monoid morphism
$V(R)\longrightarrow V(R/\mathbf{a})$ is surjective.  Consider the
following assertions:

\begin{enumerate}
\item $\mathcal{H}_\mathbf{t}$ is a module category;
\item There is an epimorphism $\xymatrix{M\ar@{>>}[r]&\mathbf{a}}$, where $M$ is a finitely generated projective
$R/t(R)$-module which is in $\mathcal{C}$.

\item $\mathbf{a}$ is the trace of some finitely generated projective
left $R$-module;

\item $\mathbf{t}$ is an HKM torsion pair;

\item
$\mathcal{H}_\mathbf{t}$  has a progenerator which is a classical
tilting complex.
\end{enumerate}
Then the implications $5)\Longrightarrow 4)$ and $3)\Longrightarrow
4)\Longrightarrow 1)\Longleftrightarrow 2)$ hold true. When the
monoid morphism $V(R)\longrightarrow V(R/t(R))$ is also surjective,
all assertions are equivalent.
\end{cor}
\begin{proof}
$3)\Longrightarrow 4)\Longrightarrow 1)$ and $5)\Longrightarrow 4)$
follow from corollaries \ref{cor.torsion theory given by projective}
and  \ref{cor.classical tilting progenerator}.

$1)\Longrightarrow 2)$ Consider the finitely generated projective
module $P$ and the exact sequence \linebreak $\xymatrix{0\ar[r] & F
\ar[r] & C\ar[r]& \mathbf{a}P \ar[r] & 0}$ given by theorem
\ref{teor.modular H with easy conditions}. It follows that
$C\in\mathcal{F}$ since so do $F$ and $\mathbf{a}P$. But, then, the
fact that $\text{Ext}_R^1(C,?)_{|\mathcal{F}}=0$ implies that $C$ is
a finitely generated projective $R/t(R)$-module. Since $C$ generates
$\mathcal{C}\cap\mathcal{F}$ we get an epimorphism
$\xymatrix{M:=C^n\ar@{>>}[r]&\mathbf{a}}$ as desired.

$2)\Longrightarrow 1)$  is a direct consequence of theorem
\ref{teor.modular H with easy conditions}.

$2)\Longrightarrow 3), 5)$ (Assuming that the monoid map
$\xymatrix{V(R)\ar[r] &V(R/t(R))}$ is surjective). We have a
finitely generated projective $R$-module $Q$ such that
$M\cong\frac{Q}{t(R)Q}=\frac{Q}{t(Q)}$. But we then get
$Q=\mathbf{a}Q\oplus t(Q)$ since $\mathbf{a}M=M$ and $\mathbf{a}Q$
is in $\mathcal{F}$. It follows that $M\cong\mathbf{a}Q$ is also
projective as an $R$-module. Then we have
$\text{Gen}(M)=\text{Gen}(\mathbf{a})$, so that
$\mathbf{a}=\text{tr}_M(R)$.

 On the other hand, by taking $C=M$ in theorem \ref{teor.modular H with easy
conditions}, we know that the complex

\begin{center}
$\xymatrix{G:=\cdots \ar[r] & 0\ar[r]& M\oplus
M\ar[rr]^{\hspace{0.4cm}\begin{pmatrix} jq & 0
\end{pmatrix}}&&P\ar[r]& 0\ar[r] &  \cdots }$
\end{center}
concentrated in degrees $-1$ and $0$, is a progenerator of
$\mathcal{H}_\mathbf{t}$.
\end{proof}

In  \cite[Corollary 2.13]{MT} (see also \cite[Lemma 4.1]{CMT}) the
authors proved that a faithful (not necessarily hereditary) torsion
pair in $R-\text{Mod}$ has a heart which is a module category if,
and only if, it is an HKM torsion  pair. Our next result shows that, for
hereditary torsion pairs, we can be more precise.

\begin{cor} \label{cor.faithful hereditary torsion pair}
Let $R$ be a ring  and let $\mathbf{t}=(\mathcal{T},\mathcal{F})$ be
a faithful hereditary torsion pair in $R-\text{Mod}$. Consider the
following assertions:

\begin{enumerate}
\item There is a finitely generated projective $R$-module $Q$ such
that $\mathcal{T}=\text{Ker}(\text{Hom}_R(Q,?))$ and the trace
$\mathbf{a}$ of $Q$ in $R$ is finitely generated as a left ideal;
\item $\mathcal{H}_\mathbf{t}$ is a module category;
\item $\mathbf{t}$ is an HKM torsion pair;
\item $\mathcal{H}_\mathbf{t}$ has a progenerator which is a
classical tilting complex;
\item There is an idempotent ideal $\mathbf{a}$ of $R$ which satisfies the following properties:

\begin{enumerate}
\item $\mathbf{t}$ is the right constituent torsion pair of the TTF
triple defined by $\mathbf{a}$;
\item there is a progenerator $V$ of $R/\mathbf{a}-\text{Mod}$ which
admits a finitely generated projective resolution \newline
$\xymatrix{Q\ar[r]^{d}&P\ar[r] &V\ar[r] &0}$ in $R-\text{Mod}$
satisfying the following two properties:

\begin{enumerate}
\item $\text{Ker}(d)\subseteq\mathbf{a}Q$;
\item there is a morphism
$\xymatrix{Q^{(J)}\ar[r]^{h}&\mathbf{a}}$, for some set $J$, such
that $\xymatrix{h_{|
\text{Ker}(d)^{(J)}}:\text{Ker}(d)^{(J)}\ar[r]&\mathbf{a}}$ is an
epimorphism.
\end{enumerate}
\end{enumerate}
\end{enumerate}
Then the implications $1)\Longrightarrow 2)\Longleftrightarrow
3)\Longleftrightarrow 4)\Longleftrightarrow 5)$ hold true. When the
monoid morphism \linebreak $\xymatrix{V(R) \ar[r] & V(R/I)}$ is
surjective, for all idempotent two-sided ideals $I$ of $R$, all
assertions are equivalent.
\end{cor}
\begin{proof}
$1)\Longrightarrow 2)$ follows from corollary \ref{cor.torsion
theory given by projective}.

$2)\Longrightarrow 4)$ follows from  \cite[Lemma 4.1]{CMT}.

$4)\Longrightarrow 3)$ is a consequence of corollary
\ref{cor.classical tilting progenerator}.

$3)\Longrightarrow 2)$ follows from \cite[Theorems 2.10 and
2.15]{HKM}.

$2),4)\Longrightarrow 5)$ By corollary \ref{cor.faithful hereditary
pair}, we know that $\mathbf{t}$ is the right constituent of a TTF
triple in $R-\text{Mod}$ defined by an idempotent ideal $\mathbf{a}$
which is finitely generated on the left. Let now fix a classical
tilting complex $\xymatrix{G:= \cdots\ar[r] &0\ar[r]& Q\ar[r]^{d}
&P\ar[r] & 0 \ar[r] & \cdots}$ which is a progenerator of
$\mathcal{H}_\mathbf{t}$. Assertion 5 follows by taking $V=H^0(G)$
and by applying theorem \ref{teor.hereditary case} to $G$.

$5)\Longrightarrow 3)$ follows by taking the complex
$\xymatrix{P^{\bullet}=G:= \cdots\ar[r] &0\ar[r]& Q\ar[r]^{d}
&P\ar[r] & 0 \ar[r] & \cdots}$, concentrated in degrees $-1$ and
$0$, and applying corollary \ref{cor.classical tilting
progenerator}.

$2)\Longrightarrow 1)$ (Assuming that $\xymatrix{V(R)\ar[r]
&V(R/I)}$ is surjective, for all two-sided ideals $I$ of $R$). It is
a consequence of corollary \ref{cor.when a is torsionfree}.
\end{proof}

\begin{cor} \label{cor.split and commutative case}
Let $\mathbf{a}$ be a two-sided idempotent ideal of $R$ whose
associated TTF triple is left split and put
$\mathbf{t}=(\mathcal{T},\mathcal{F})$. Then
$\mathcal{H}_\mathbf{t}$ is equivalent to $\frac{R}{t(R)}\times
\frac{R}{\mathbf{a}}-\text{Mod}$. When the TTF triple is centrally split, then
$\mathcal{H}_\mathbf{t}$ is equivalent to $R-\text{Mod}$.
\end{cor}
\begin{proof}
In this case $\mathbf{a}$ is a direct summand of $_RR$, whence
projective,  so that example \ref{exem.sum of stalks} and
proposition \ref{prop.progenerator as sum of stalks} apply.  Note
that then $V=\frac{R}{\mathbf{a}}$ is a projective left $R$-module,
which implies that $\mathcal{H}_\mathbf{t}$ is equivalent to
$S-\text{Mod}$, where $S\cong
\text{End}_R(\frac{R}{t(R)})^{op}\times\text{End}_R(\frac{R}{\mathbf{a}})^{op}\cong\frac{R}{t(R)}\times
\frac{R}{\mathbf{a}}$.

When the TTF triple is centrally split, we have a central idempotent
$e$ such that $\mathbf{a}=Re$ and $t(R)=R(1-e)$, so that
$R/\mathbf{a}\cong R(1-e)$ and $R/t(R)\cong Re$. The result in this
case follows immediately since we have a ring isomoprhism $R\cong
Re\times R(1-e)$.
\end{proof}

\section{Some examples}

We now give a few examples which illustrate the results obtained in
the previous sections. All of them refer to finite dimensional
algebras over a field which are given by quivers and relations. We
refer the reader to \cite[Chapter II]{ASS} for the terminology that
we use.

\begin{exem} \label{exem.counterexample to Mantese-Tonolo}
Let $\xymatrix{Q_n:1\ar[r] &2 \ar[r]&\cdots \ar[r] & n}$ ($n>1$) be
the Dynkin quiver of type $\mathbf{A}$,  let $R=KQ_n$ the
corresponding path algebra, where $K$ is any field, and let us take
$\mathbf{a}=Re_nR$. If $(\mathcal{C},\mathcal{T},\mathcal{F})$ is
the TTF triple associated to $\mathbf{a}$ and
$\mathbf{t}=(\mathcal{T},\mathcal{F})$ is its right constituent
torsion pair, then $\mathcal{H}_\mathbf{t}$ is equivalent to
$K\times KQ_{n-1}-\text{Mod}$. In particular $\mathcal{D}(R)$ and
$\mathcal{D}(\mathcal{H}_\mathbf{t})$ are not equivalent
triangulated categories.

\end{exem}
\begin{proof}
We  have that $\mathbf{a}=Re_n$ and this shows that $e_nR(1-e_n)=0$
and that $\mathbf{a}$ is projective and injective as left
$R$-module. Since $R$ is hereditary we conclude that $\mathcal{C}$
consists of injective modules. Therefore
$(\mathcal{C},\mathcal{T},\mathcal{F})$ is a left split TTF triple
(see also \cite[Theorem 3.1]{NS}).

On the other hand, we have that $t(R)=R(1-e_n)\oplus Je_n$, where
$J$ is the Jacobson radical of $R$. This implies that $R/t(R)\cong
K$. On the other hand, we clearly have that $R/\mathbf{a}$ is
isomorphic to the path algebra $ KQ_{n-1}$  of type
$\mathbf{A}_{n-1}$. Then, by corollary \ref{cor.split and
commutative case}, we have $\mathcal{H}_\mathbf{t}\cong K\times
KQ_{n-1}-\text{Mod}$. Moreover  $R=KQ_n$ and $K\times KQ_{n-1}$
cannot be derived equivalent algebras  because their centers
(=$0$-th Hochschild cohomology spaces) are not isomorphic (see
\cite[Proposition 2.5]{R2}).
\end{proof}

\begin{exem}
Let $K$ be a field and $R$ be the $K$-algebra given by the following
quiver and relations:

$$\xymatrix{ 1 \ar@<1ex>[r]^{\alpha} \ar@<-0.5ex>[r]_{\beta}
& 2  \ar@<1ex>[r]^{\gamma} \ar@<-0.5ex>[r]_{\delta} & 3 } \hspace{2
cm} \alpha \delta\text{, } \beta\gamma \text{ and
}\alpha\gamma\text{-}\beta\delta$$

Let $\mathbf{a}$ be an idempotent ideal of $R$, let
$(\mathcal{C},\mathcal{T},\mathcal{F})$ be the associated TTF triple
in $R-\text{Mod}$ and let $\mathbf{t}=(\mathcal{T},\mathcal{F})$ be
its right constituent torsion pair. The following facts are true:

\begin{enumerate}
\item If $\mathbf{a}=Re_1R$ then $\mathcal{H}_\mathbf{t}$ has a progenerator which is a sum of stalk complexes and
 $\mathcal{H}_\mathbf{t}\cong K\Gamma-\text{Mod}$, where $\Gamma$ is the quiver $\xymatrix{2
\ar@<1ex>[r] \ar@<-1ex>[r]& 3 \ar[r] & 1}$.  The algebra $R$ is then
tilted of type $\Gamma$.
\item If $\mathbf{a}=Re_2R$ then $\mathcal{H}_\mathbf{t}$ is
equivalent to $K\times K\times K-\text{Mod}$.
\item If $\mathbf{a}=R(e_1+e_2)R$ then $\mathcal{H}_\mathbf{t}$ does not have a sum of stalk complexes as a progenerator and
$\mathcal{H}_\mathbf{t}\cong S-\text{Mod}$, where $S$ is the algebra
given by the following quiver and relations

$$\xymatrix{ 2 \ar@<0ex>[rr]^{\mu_2} \ar@<2.2ex>[rr]^{\mu_1} \ar@<-2.2ex>[rr]^{\mu_3}  &&  3  \ar@<1ex>[rr]^{\pi_1} \ar@<-1ex>[rr]_{\pi_2} && 1}$$
\begin{center}

$\mu_1\pi_2=\mu_3\pi_1=0$;

$\mu_1\pi_1=-\mu_2\pi_2$

$\mu_2\pi_1=\mu_3\pi_2$.
\end{center}
The algebras $R$ and $S$ are derived equivalent and, hence, $S$ is
piecewise hereditary (i.e. derived equivalent to a hereditary
algebra).
\end{enumerate}
\end{exem}
\begin{proof}
Using a classical visualization of modules via diagrams (see, e.g.,
\cite{F}), the indecomposable projective left $R$-modules can be
depicted as:

$$
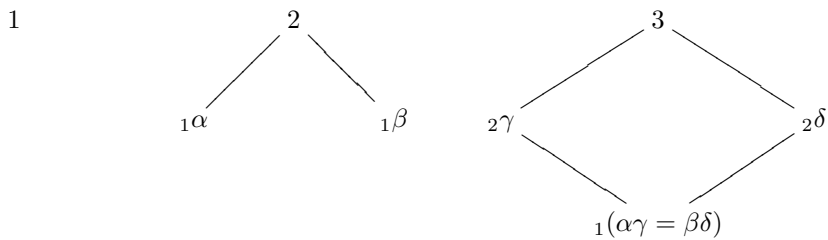
\xymatrix{1 & & & 2 \ar@{-}[ld] \ar@{-}[rd] &  & & 3 \ar@{-}[dr] \ar@{-}[dl] & \\ && _{1}\alpha & & _{1}\beta &  _{2}\gamma \ar@{-}[dr] && _{2}\delta \ar@{-}[dl] &&&\\ &&&&&& _{1}(\alpha \gamma=\beta \delta)&&}$$

By corollary \ref{cor.faithful hereditary torsion pair}, whenever
$\mathbf{t}$ is faithful, if $\mathcal{H}_\mathbf{t}$ is equivalent
to $S-\text{Mod}$, then $S$ and $R$ are derived equivalent. Even
more, in that case $R[1]$ is a tilting object of
$\mathcal{H}_\mathbf{t}$, which implies that $R$ and $S$ are
tilting-equivalent.

1) In this case we have  $\mathbf{a}=\text{Soc}(_RR)\cong
S_1^{(4)}\cong Re_1^{(4)}$, so that $\mathbf{a}$ is projective in
$R-\text{Mod}$ and $\mathbf{t}$ is a faithful torsion pair. Then we
have that $\text{Ext}_R^2(\frac{R}{\mathbf{a}},?)\equiv 0$ and, by
example \ref{exem.sum of stalks}, we conclude that
$\frac{R}{\mathbf{a}}\oplus\mathbf{a}[1]$ is a progenerator of
$\mathcal{H}_\mathbf{t}$. It follows that
$G:=\frac{R}{\mathbf{a}}\oplus S_1[1]$ is also a progenerator of
$\mathcal{H}_\mathbf{t}$.

We put $A:=\frac{R}{\mathbf{a}}$, which is isomorphic to the
Kronecker algebra: $\xymatrix{2  \ar@<1ex>[r] \ar@<-0.5ex>[r] & 3 }$
and which we also view as a left $R$ module annihilated by
$\mathbf{a}$. Then $\mathcal{H}_\mathbf{t}$ is equivalent to
$S-\text{Mod}$, where
$S=\text{End}_{\mathcal{H}_\mathbf{t}}(G)^{op}\cong\begin{pmatrix}
\text{End}_A(S_1)^{op} & 0\\ \text{Ext}_R^1(A,S_1) & A
\end{pmatrix}\cong\begin{pmatrix}
K & 0\\ \text{Ext}_R^1(A,S_1) & A
\end{pmatrix}$.

Note that $A=S_2\oplus\frac{Re_3}{R\alpha\gamma}$ as left
$R$-module, so that we have a vector space decomposition
$\text{Ext}_R^1(A,S_1)\cong\text{Ext}_R^1(S_2,S_1)\oplus\text{Ext}_R^1(\frac{Re_3}{R\alpha\gamma},S_1)$.
Let $\epsilon'$ be the element of $\text{Ext}_R^1(A,S_1)$
represented by the short exact sequence

\begin{center}
 $\xymatrix{0 \ar[r] & S_1 \hspace{0.08cm}\ar@{^(->}[r] &
Re_3\ar[r]&\frac{Re_3}{R\alpha\gamma}\ar[r]& 0}$
\end{center}
The assignment $a\rightsquigarrow a\epsilon'$ gives an isomorphism
of left $A$-modules
$\xymatrix{Ae_3\ar[r]^{\sim\hspace{0.7cm}}&\text{Ext}_R^1(A,S_1)}$,
so
that the algebra $S$ is isomorphic to $\begin{pmatrix} K & 0\\
Ae_3 & A
\end{pmatrix}\cong K\Gamma$, where $\Gamma$ is
the quiver $\xymatrix{2 \ar@<1ex>[r] \ar@<-1ex>[r]& 3 \ar[r] & 1}$.

\vspace*{0.3cm}

2) We have that $\mathbf{a}=Re_2R=Re_2\oplus Je_3$ is not projective
 and that $R/Re_2R\cong S_1\oplus S_3$ in $R-\text{Mod}$. We then
 get that $\mathcal{F}=\{F\in R-\text{Mod}:$
 $\text{Soc}(F)\in\text{Add}(S_2)\}$, so that $\mathbf{t}$ is not
 faithful. The minimal projective resolution of $Je_3$ is of the form $\xymatrix{0\ar[r]& P_1^{(3)}\ar[r] & P_2^{(2)}\ar[r]&Je_3\ar[r]&
 0}$, where $P_1=S_1$ is simple projective. It follows that

\begin{center}
 $\text{Ext}_R^2(\frac{R}{\mathbf{a}},F)\cong\text{Ext}_R^2(S_1\oplus
 S_3,F)\cong\text{Ext}_R^2(S_3,F)\cong\text{Ext}_R^1(Je_3,F)=0$,
 \end{center}
 for each $F\in\mathcal{F}$. By example \ref{exem.sum of stalks},
 we know that
 $\frac{R}{\mathbf{a}}[0]\oplus\frac{\mathbf{a}}{t(\mathbf{a})}[1]\cong (S_1\oplus S_3)[0]\oplus S_1^{(3)}[1]$
 is a progenerator of $\mathcal{H}_\mathbf{t}$, which implies that
 $G:=(S_1\oplus S_3)[0]\oplus S_1[1]$ is also a progenerator of
 $\mathcal{H}_\mathbf{t}$. Since we have
 $\text{Ext}_R^1(S_1,S_1)=0=\text{Ext}_R^1(S_3,S_1)$, we get from
 proposition \ref{prop.progenerator as sum of stalks} that
 $\mathcal{H}_\mathbf{t}$ is equivalent to $S-\text{Mod}$, where $S=\text{End}_R(S_1)^{op}\times\text{End}_R(S_1\oplus S_3)^{op}\cong K\times K\times
 K$.

\vspace*{0.3cm}

3) We have isomorphisms $\mathbf{a}=R(e_1+e_2)R\cong Re_1\oplus
Re_2\oplus Je_3$ and $R/\mathbf{a}\cong S_3$ in $R-\text{Mod}$, and
this implies that $\mathcal{F}=\{F\in R-\text{Mod}:$
 $\text{Soc}(F)\in\text{Add}(S_1\oplus S_2)\}$ and that $\mathcal{F}$ is
 faithful. Since $\text{Ext}_R^2(S_3,S_1)\neq
 0$, this time we do not have a sum of stalk complexes as a
 progenerator of $\mathcal{H}_\mathbf{t}$. Instead, inspired by the proof of corollary \ref{cor.when a is torsionfree}, we consider the
 minimal projective resolution of $S_3\cong R/\mathbf{a}$

 \begin{center}
 $\xymatrix{0\ar[r]&P_1^{(3)}\ar[r]&P_2^{(2)}\ar[r]^{d'}&P_3\ar[r]&S_3 \ar[r] & 0}$
 \end{center}
 and take the complex $\xymatrix{G':= \cdots \ar[r] & 0 \ar[r] &  P_2^{(2)}\ar[r]^{d'}&P_3\ar[r] &  0 \ar[r] & \cdots}$, concentrated in degrees $-1$ and $0$. Now the complex $G:=G'\oplus
 P_1[1]\oplus P_2[1]$ satisfies all conditions in assertion 2 of corollary \ref{cor.classical tilting
 progenerator}, so that it is a progenerator of
 $\mathcal{H}_\mathbf{t}$ and $\mathcal{H}_\mathbf{t}\cong
 S-\text{Mod}$, where $S=\begin{pmatrix} \text{End}_R(P_1\oplus P_2)^{op} & \text{Hom}_{\mathcal{D}(R)}(P_1[1]\oplus P_2[2],G')
 \\ \text{Hom}_{\mathcal{D}(R)}(G',P_1[1]\oplus P_2[2]) &
 \text{End}_{\mathcal{D}(R)}(G')^{op} \end{pmatrix}$.

We clearly have $\text{End}_R(P_1\oplus
P_2)^{op}\cong\begin{pmatrix} \text{End}_R(P_2)^{op} & 0\\
\text{Hom}_R(P_1,P_2) & \text{End}_R(P_1)^{op}\end{pmatrix}\cong \begin{pmatrix} K & 0\\
K^2 & K\end{pmatrix}=:A$, which is isomorphic to the Kronecker
algebra. Moreover, the $0$-homology functor defines an isomorphism
\linebreak
$\xymatrix{\text{End}_{\mathcal{D}(R)}(G')\ar[r]^{\sim\hspace{1.6
cm}}&\text{End}_R(H^0(G'))\cong\text{End}_R(S_3)\cong K}$. On the
other hand, $\text{Hom}_{\mathcal{D}(R)}(G',P_1[1]\oplus P_2[1])$ is
a $2$-dimensional vector space,  where a basis $\{\pi_1,\pi_2\}$ is
induced by the two projections $\xymatrix{P_2^{(2)}\ar@{>>}[r]&
P_2}$. Similarly, $\text{Hom}_{\mathcal{D}(R)}(P_1[1]\oplus
P_2[2],G')$ is a 3-dimensonal vector space with a basis $\{\mu_i:$
$i=1,2,3\}$ induced by the monomorphisms $\xymatrix{P_1=Re_1
\hspace{0.07cm}\ar@{^(->}[r]&P_2^{(2)}}$ which map $e_1$ onto
$(\beta ,0)$, $(\alpha ,-\beta )$ and $(0,\alpha )$, respectively.
Since the multiplication in   $S$ is given by anti-composition of
the entries, we easily get that $\pi_i\mu_j=\mu_j\circ\pi_i=0$, for
all $i,j$. On the other hand, we have:

\begin{center}

 $\mu_3\pi_1=\pi_1\circ\mu_3=0=\pi_2\circ\mu_1=\mu_1\pi_2$

$\mu_1\pi_1=\pi_1\circ\mu_1=-\pi_2\circ\mu_2=-\mu_2\pi_2=\begin{pmatrix}
0 & 0\\ \rho_\beta & 0
\end{pmatrix}:\xymatrix{P_1\oplus P_2\ar[r]&P_1\oplus P_2}$

$\mu_2\pi_1=\pi_1\circ\mu_2=\pi_3\circ\mu_2=\mu_2\pi_3=\begin{pmatrix}
0 & 0\\ \rho_\alpha & 0
\end{pmatrix}:\xymatrix{P_1\oplus P_2\ar[r]&P_1\oplus P_2}$
\end{center}
where $\xymatrix{\rho_x:P_1=Re_1\ar[r] & P_2}$ maps
$a\rightsquigarrow ax$, for each $x\in P_2$. It easily follows that
$S$ is given by quivers and relations as claimed in the statement.
\end{proof}

Given a finite quiver $Q$ with no oriented cycle, a path $p$ will be
called a \emph{maximal path} when its origin is a source and its
terminus is a sink. We put
$D=\text{Hom}_K(?,KQ)=KQ-\text{mod}^{op}\stackrel{\cong}{\longleftrightarrow}\text{mod}-KQ$
to denote the usual duality between finitely generated left and
right $KQ$-modules.

\begin{exem} \label{exem.progenerator V not tilting}
Let $Q$ be a finite connected quiver with no oriented cycles which
is different from $1\rightarrow \dots \rightarrow n$, and let $i\in
Q_0$ be a source. Let us form a new quiver $\hat{Q}$ as follows. We
put $\hat{Q}_0=Q_0$ and the arrows of $\hat{Q}$ are the arrows of
$Q$ plus an arrow $\alpha_p:t(p)\rightarrow i$, for each maximal
path $p$ in $Q$. Given a field $K$, we consider the $K$-algebra $R$
with quiver $\hat{Q}$ and relations:

\begin{enumerate}
\item $\alpha_p\beta =0$, for each  $\beta\in Q_1$ and each maximal path $p$ in $Q$;

\item $p'\alpha_p=q'\alpha_q$, whenever $p'$ and $q'$  are paths in $Q$ such
that $s(p')=s(q')$ and there is a path $\pi:j\rightarrow \dots
\rightarrow s(p')=s(q')$ in $Q$ such that $\pi p'=p$ and $\pi q'=q$.
\end{enumerate}
We identify $KQ-\text{Mod}$ with the full subcategory of
$R-\text{Mod}$ consisting of the $R$-modules annihilated by the
two-sided ideal generated by the $\alpha_p$. Then
$\mathbf{t}=(KQ-\text{Inj},(KQ-\text{Inj})^\perp)$ is a non-tilting
torsion pair in $R-\text{Mod}$ such that $D(KQ)[0]$ is a
progenerator of $\mathcal{H}_\mathbf{t}$.
\end{exem}
\begin{proof}
Let $\mathcal{M}$ be the set of maximal path in $Q$ and consider the
paths of $Q$ as the canonical basis $B$ of $KQ$. Its dual basis is
denoted by $B^*$. Consider the assignment
$\sum_{p\in\mathcal{M}}a_p\alpha_p\rightsquigarrow
(\sum_{p\in\mathcal{M}}a_pp^*)\otimes\bar{e}_i$, where $a_p\in
KQe_{t(p)}$ for each $p\in\mathcal{M}$ and
$\bar{e}_i=e_i+e_iJ\in\frac{e_iKQ}{e_iJ}$ is the canonical element.
Here $J=J(KQ)$ is the Jacobson radical,  a basis of which is given
by the paths of length $>0$. This assignment defines an isomorphism
of $KQ$-bimodules

\begin{center}
$\mathbf{a}:=\sum_{p\in\mathcal{M}}R\alpha_pR=\sum_{p\in\mathcal{M}}R\alpha_p\stackrel{\cong}{\longrightarrow}D(KQ)\otimes_K\frac{e_iKQ}{e_iJ}=:M$.
\end{center}
Moreover, it is the restriction of an algebra isomorphism
$R\stackrel{\cong}{\longrightarrow}KQ\rtimes M$ which maps
$e_i\rightsquigarrow (e_i,0)$, $\beta\rightsquigarrow (\beta ,0)$
and   $\alpha_p\rightsquigarrow (0,p^*\otimes\bar{e}_i)$, for all
$i\in Q_0$, all $\beta\in Q_1$ and all $p\in\mathcal{M}$.

Our hypotheses on $Q$ guarantee that there is no
projective-injective $KQ$-module. Then $D(KQ)$ is a classical
1-tilting $KQ$-module whose associated torsion pair in
$KQ-\text{Mod}$, namely
$\mathbf{t}=(KQ-\text{Inj},(KQ-\text{Inj})^\perp)$,  is faithful. On
the other hand, the fact that $i$ is a source and $Q$ is connected
implies that $\frac{e_iKQ}{e_iJ}\otimes_{KQ}D(KQ)=0$. Then theorem
\ref{teor.nontilting pair with stalk progenerator} applies, with
$A=KQ$, $V=D(KQ)$ and $X=\frac{e_iKQ}{e_iJ}$.
\end{proof}

\end{document}